\documentclass[12pt,notitlepage]{amsart}
\usepackage{latexsym,amsfonts,amssymb,amsmath,amsthm}
\usepackage{graphicx}
\usepackage{color}
\usepackage[normalem]{ulem}
   \renewcommand{\a}{\mathfrak{a}}
\renewcommand{\b}{\mathfrak{b}}
\let\turc\c
\usepackage{etoolbox}
\robustify\turc 
\renewcommand{\c}{\mathfrak{c}}
\newcommand{\bpm}{\begin{pmatrix}}
\newcommand{\epm}{\end{pmatrix}}
\newcommand{\Z}{\ensuremath{\mathbb{Z}}}
\newcommand{\inv}{^{-1}}
\DeclareMathOperator{\SL}{SL}

\usepackage[inner=1.0in,outer=1.0in,bottom=1.0in, top=1.0in]{geometry}

\newcommand{\mz}{\ensuremath{\mathbb Z}}

\newcommand{\mh}{\ensuremath{\mathbb H}}

\newcommand{\mc}{\ensuremath{\mathbb C}}

\newcommand{\mymod}{\ensuremath{\negthickspace \negmedspace \pmod}}
\newcommand{\shortmod}{\ensuremath{\negthickspace \negthickspace \negthickspace \pmod}}

\newcommand{\intR}{\int_{-\infty}^{\infty}}

\newcommand{\sumstar}{\sideset{}{^*}\sum}
\newcommand{\sumprime}{\sideset{}{'}\sum}

\DeclareMathOperator{\sgn}{sgn}

\theoremstyle{plain}		
	\newtheorem{mytheo}{Theorem} [section]
	
	\newtheorem{myprop}[mytheo]{Proposition}
	\newtheorem{mycoro}[mytheo]{Corollary}
     \newtheorem{mylemma}[mytheo]{Lemma}

\theoremstyle{remark}

\numberwithin{equation}{section}
\numberwithin{figure}{section}

\begin{document}
\title{Explicit calculations with Eisenstein series}
\author{
Matthew P. Young}

\address{Department of Mathematics \\
	  Texas A\&M University \\
	  College Station \\
	  TX 77843-3368 \\
		U.S.A.}
\email{myoung@math.tamu.edu}

\thanks{This material is based upon work 
supported by the National Science Foundation under agreement No. DMS-1401008.  Any opinions, findings and conclusions or recommendations expressed in this material are those of the authors and do not necessarily reflect the views of the National Science Foundation. 
 }


  \begin{abstract}
   We find explicit change-of-basis formulas between Eisenstein series attached to cusps, and newform Eisenstein series attached to pairs of primitive Dirichlet characters.  
As a consequence, we prove a Bruggeman-Kuznetsov formula for newforms of square-free level and trivial nebentypus.  

We also derive, in explicit form, the Fourier expansion, Hecke eigenvalues, Atkin-Lehner pseudo-eigenvalues, and other data associated to these Eisenstein series, with arbitrary integer weight, level, and nebentypus.   
   
  \end{abstract}

 \maketitle
\section{Introduction} 
 Let $N$ be a positive integer, and consider the space of automorphic forms of level $N$, weight $k \in \mathbb{Z}$, and nebentypus $\psi$ modulo $N$.  There are at least two natural choices of how to decompose the space spanned by the Eisenstein series.  One is to use Eisenstein series $E_{\a}(z,s, \psi)$ attached to cusps 
 and the
 other is to use Eisenstein series $E_{\chi_1, \chi_2}(z,s)$ attached to pairs of Dirichlet characters, 
 which has a natural interpretation from representation theory.  The decomposition along cusps is quite convenient from the point of view of the spectral decomposition and the Parseval formula since in essence Eisenstein series attached to inequivalent cusps are orthogonal.  The decomposition with Dirichlet characters is friendly when studying the $L$-functions associated to automorphic forms.  Clearly, these two features are in conflict with each other.
 
 One of the main goals of this paper is to explicitly work out the translations between these different bases, and to use this information to derive some other useful formulas.  These change of basis formulas appear as Theorems \ref{thm:EcuspInTermsofEchichi} and \ref{thm:EchichiInTermsofEa}.  
 
 Along the way, we derive the Fourier expansion (see Proposition \ref{prop:FourierExpansion}), allowing us to derive the functional equation of $E_{\chi_1, \chi_2}(z,s)$ under $s \rightarrow 1-s$.  The Fourier expansion also shows that $E_{\chi_1,\chi_2}$ is an eigenfunction of all the Hecke operators, and gives explicit formulas for the Hecke eigenvalues.
  We additionally examine the
Mellin transform of the Eisenstein series $E_{\chi_1, \chi_2}(iy,s)$, leading to a complementary functional equation related to the functional equation of the Dirichlet $L$-functions.

After developing the change of basis formulas, in Section \ref{section:orthogonality} we examine the orthogonality properties of the various types of Eisenstein series.  

In Section \ref{section:AtkinLehnerEigenvalues}, we explicitly derive the action of the Atkin-Lehner operators.  

As a culmination of all this work, in Section \ref{section:Kuznetsov} we derive a Bruggeman-Kuznetsov formula restricted to newforms, in the special case where the level is squarefree, the nebentypus is principal, and the weight is $0$.  This is an extension of the derivation of the Petersson formula for newforms of \cite{PetrowYoung}, which proceeds by a sieving argument.  The method of \cite{PetrowYoung} may be easily modified to sieve for cuspidal Maass newforms, but unfortunately the same approach does not immediately carry over to sieve the continuous spectrum.  This sieving argument is a requirement to prove the Bruggeman-Kuznetsov formula for newforms.  The material developed in this paper may be used to sieve the Eisenstein series in a close analogy to the cusp forms.  

In \cite{PetrowYoung}, the newform Petersson formula is the crucial intitial step to set up the cubic moment which is in turn used to prove a strong subconvexity bound for twisted $L$-functions $L(f \otimes \chi_q, 1/2)$ where $f$ is a holomorphic newform of squarefree level $N$, and $\chi_q$ is a quadratic character of conductor $q$.  This paper generalized the groundbreaking work of Conrey and Iwaniec \cite{ConreyIwaniec}, which required $N|q$.  With the aid of the newform Bruggeman-Kuznetsov formula, the tools are now in place to allow for $f$ to be a Maass newform.  The arithmetical aspects of \cite{PetrowYoung} will be the same, but some of the analytical aspects (e.g., integral transforms with Bessel functions) will look somewhat different.  Conrey and Iwaniec \cite{ConreyIwaniec} successfully treated both the holomorphic forms and Maass forms in tandem, so 
 it is reasonable to expect that the subconvexity bound of \cite{PetrowYoung} may be extended to Maass forms.  A motivation for this extension is the hybrid equidistribution of Heegner points as both the level and discriminant vary (as in \cite{LMY}); such a result would immediately improve many of the exponents in \cite{LMY}.
 
 

 An overarching goal of this paper is to provide, in one place, and in explicit form, many of the basic properties of Eisenstein series, using the classical language.  Many special cases (e.g., with principal nebentypus, or primitive nebentypus, or weight $k=0$, or with holomorphic forms, \dots) are scattered throughout the literature, and the author found \cite{Huxley} \cite{IwaniecClassical} \cite{IwaniecSpectral} \cite{DFI} \cite{DiamondShurman} \cite{KnightlyLi} particularly useful references.
 

\section{Acknowledgments}
I thank Peter Humphries and Jiakun Pan for useful feedback and corrections on earlier drafts of this paper, and especially E. Mehmet Kiral for many conversations on this work.  I also thank Andrew Booker and Min Lee for informing me of their independent work with Andreas St\"ombergsson on the change-of-basis formulas, and Gergely Harcos, Abhishek Saha, and Rainer Schulze-Pillot for conversations on orthogonalization.

 \section{Definitions}
 \subsection{Eisenstein series attached to cusps}
 \label{section:EisensteinCusps}
 Let $\Gamma = \Gamma_0(N)$, suppose $\psi$ is a Dirichlet character modulo $N$, and let $k \in \mathbb{Z}$.  We study automorphic functions $f$ on $\Gamma$ of weight $k$ and nebentypus $\psi$, which satisfy the transformation formula
 \begin{equation*}
  f(\gamma z) = j(\gamma, z)^k \psi(\gamma) f(z),
 \end{equation*}
where $\psi(\gamma) = \psi(d)$, with $d$ denoting the lower-right entry of $\gamma$, and where
\begin{equation*}
 j(\gamma, z) = \frac{cz+d}{|cz+d|}.
\end{equation*}
Since $-I \in \Gamma_0(N)$, a necessary condition for a nonzero automorphic function to exist is that
\begin{equation*}
 \psi(-1) = (-1)^k.
\end{equation*}

  Let $\a$ be a cusp for $\Gamma$, and let $\sigma_{\a}$ be a scaling matrix for $\a$, which means $\sigma_{\a} \infty = \a$, and $\sigma_{\a}^{-1} \Gamma_{\a} \sigma_{\a} = \Gamma_{\infty}= \{\pm (\begin{smallmatrix} 1 & b \\ & 1 \end{smallmatrix} ): b \in \mz \}$.  Let $\tau_{\a} = \sigma_{\a} (\begin{smallmatrix} 1 & 1 \\ & 1 \end{smallmatrix})  \sigma_{\a}^{-1}$, so that $\pm \tau_{\a}$ generate $\Gamma_{\a}$, 
 the stablizer of $\a$ in $\Gamma$.  We say that
$\a$ is \emph{singular} for $\psi$  if $\psi(\tau_{\a}) = 1$.
 The Eisenstein series of nebentypus $\psi$ and weight $k$ attached to the cusp $\a$ is defined by
 \begin{equation}
 \label{eq:Eadef}
  E_{\a}(z,s, \psi) = \sum_{\gamma \in \Gamma_{\a} \backslash \Gamma} \overline{\psi}(\gamma) j(\sigma_{\a}^{-1} \gamma, z)^{-k} (\text{Im } \sigma_{\a}^{-1} \gamma z)^s,
 \end{equation}
 initially for $\text{Re}(s) > 1$.  Since we generally have no need to combine Eisenstein series with different weights, we suppress the weight in the notation.  
 
 The definition of singular given above disagrees with 
 some definitions in the literature,
 which seemingly allowed $\psi(\tau_{\a})$ to be $-1$ for $k$ odd.  Under this assumption, $E_{\a}(z,s,\psi)$ would be ill-defined, because the summand would change sign under $\gamma \rightarrow \tau_{\a} \gamma$. 
 
 One may check that $E_{\a}$ is independent of the choice of scaling matrix.  Moreover, if $\a$ and $\b = \gamma \a$ are $\Gamma$-equivalent cusps, then
 \begin{equation}
 \label{eq:EaEquivalentCuspsRelation}
  E_{\gamma \a}(z,s,\psi) = \overline{\psi}(\gamma) E_{\a}(z,s,\psi),
 \end{equation}
 correcting a remark of \cite[p.505]{DFI}.

\subsection{Eisenstein series attached to characters}
We draw inspiration from Huxley's paper \cite{Huxley}, but refer the reader to \cite{KnightlyLi} for a more motivated definition.
Let $\chi_1, \chi_2$ be Dirichlet characters modulo $q_1, q_2$, respectively, with $\chi_1(-1)  \chi_2(-1) = (-1)^k$.  Define
\begin{equation}
\label{eq:Echi1chi2def}
 E_{\chi_1, \chi_2}(z,s) = \frac12 \sum_{(c,d) = 1} \frac{(q_2 y)^s \chi_1(c) \chi_2(d)}{|cq_2z + d|^{2s}} \Big(\frac{|c q_2 z + d|}{cq_2 z + d}\Big)^k.
\end{equation}
One can check directly that $E_{\chi_1, \chi_2}$ is automorphic on $\Gamma_0(q_1 q_2)$ of weight $k$ and nebentypus $\chi_1 \overline{\chi_2}$.  However, we can give a more structural view as follows.\footnote{The author thanks E. Mehmet Kiral for pointing out this setup.}  With $\chi_1, \chi_2$ as above, define 
\begin{equation}
\theta \begin{pmatrix} a & b \\ c & d \end{pmatrix} = \chi_1(c) \chi_2(d),
\end{equation}
for $a,b,c,d \in \mz$ and $ad-bc=1$.  Also define $\sigma = (\begin{smallmatrix} \sqrt{q_2} & \\ & 1/\sqrt{q_2} \end{smallmatrix})$, which satisfies $\sigma z = q_2 z$.  Then $E_{\chi_1, \chi_2}(z,s) = E_{\theta}(z,s)$, where $E_{\theta}$ is defined by
\begin{equation}
\label{eq:EthetaDef}
E_{\theta}(z,s) = \sum_{\gamma \in \Gamma_{\infty} \backslash \Gamma_0(1)} \theta(\gamma)
j(\gamma, \sigma z)^{-k}
 \text{Im}(\gamma \sigma z)^s .
\end{equation}
It is easy to check that $\sigma \Gamma_0(q_1 q_2) \sigma^{-1} = \Gamma_0(q_1, q_2)$, where this latter group denotes the subgroup of $SL_2(\mathbb{Z})$ with lower-left entry divisible by $q_1$, and upper-right entry divisible by $q_2$.  Moreover, if $\tau \in \Gamma_0(q_1 q_2)$, and $\tau' \in \Gamma_0(q_1, q_2)$ is defined by $\sigma \tau = \tau' \sigma$, then the lower-right entry of $\tau'$ is the same as the lower-right entry of $\tau$.  Finally, one directly checks that if $\gamma \in SL_2(\mz)$  and $\tau' \in \Gamma_0(q_1,q_2)$, then
\begin{equation}
\theta(\gamma \tau') = \theta(\gamma) (\overline{\chi_1} \chi_2)(d'),
\end{equation}
where  $d'$ is the lower-right entry of $\tau'$.  

With these properties, we may now check the automorphy of $E_{\theta}$.  For $\tau \in \Gamma_0(q_1 q_2)$, and $\tau'$ defined by $\sigma \tau = \tau' \sigma$, we have
\begin{equation}
E_{\theta}(\tau z, s) = 
\sum_{\gamma \in \Gamma_{\infty} \backslash \Gamma_0(1)} \theta(\gamma)
j(\gamma, \tau' \sigma z )^{-k}
 \text{Im}(\gamma \tau' \sigma z)^s.
\end{equation}
Changing variables by $\gamma \rightarrow \gamma \tau'^{-1}$, using
$j(\gamma \tau'^{-1}, \tau' \sigma z)^{-k} = j(\gamma, \sigma z)^{-k} j(\tau', \sigma z)^k$,
and finally $j(\tau', \sigma z) = j(\tau, z)$ (check this one directly from the definitions), we complete the proof that $E_{\theta}$ has weight $k$, nebentypus $\chi_1 \overline{\chi_2}$, and level $q_1 q_2$.

By M\"obius inversion, we have
\begin{equation}
\label{eq:Gchi1chi2Def}
 L(2s, \chi_1 \chi_2) E_{\chi_1, \chi_2}(z,s) = 
 \frac12 \sumprime_{c,d \in \mathbb{Z}} \frac{(q_2 y)^s \chi_1(c) \chi_2(d)}{|cq_2z + d|^{2s}} \Big(\frac{|c q_2 z + d|}{cq_2 z + d}\Big)^k  =: G_{\chi_1, \chi_2}(z,s),
\end{equation}
with the prime denoting that the term $c=d=0$ is omitted.
With $k=0$, Huxley's notation for our $E_{\chi_1, \chi_2}(z,s)$ is $E_{\chi_1}^{\chi_2}(q_2 z, s)$, and our $G_{\chi_1, \chi_2}(z,s)$ is Huxley's $B_{\chi_1}^{\chi_2}(q_2 z, s)$.

The reader interested in holomorphic Eisenstein series of weight $k \geq 3$ should consider 
\begin{equation}
\label{eq:EholomorphicDef}
E_{\chi_1, \chi_2, k}(z) := 
(q_2 y)^{-k/2} E_{\chi_1, \chi_2}(z, \tfrac{k}{2})
=
\frac12 \sum_{(c,d) = 1} \frac{ \chi_1(c) \chi_2(d)}{(cq_2z + d)^{k}}
,
\end{equation}
which satisfies $E_{\chi_1, \chi_2, k}(\tfrac{az+b}{cz+d}) = (cz + d)^k (\chi_1 \overline{\chi_2})(d) E_{\chi_1, \chi_2,k}(z)$, for $(\begin{smallmatrix} a& b \\ c & d \end{smallmatrix}) \in \Gamma_0(q_1 q_2)$.

\subsection{Remarks}
\label{section:newformdiscussion}
We refer the reader to \cite[Section 4]{DFI} for a more comprehensive overview of the space $L^2(\Gamma_0(N) \backslash \mh, \psi)$, including a discussion on the relevant differential operators.

Weisinger has developed a newform theory for Eisenstein series of weight $k$ and nebentypus $\psi$ in the holomorphic setting.  
The newforms of level $q_1 q_2$ are the functions 
$E_{\chi_1,\chi_2,k}(z)$  where $\chi_i$ is primitive modulo $q_i$, $i=1,2$, and $\chi_1 \overline{\chi_2} = \psi$.  The space of Eisenstein series of level $N$ is spanned by $E_{\chi_1,\chi_2,k}(Bz)$ where $B q_1 q_2  | N$.

In \cite{Huxley}, Huxley calculated the scattering matrix for the congruence subgroups $\Gamma_0(N)$, $\Gamma_1(N)$, and $\Gamma(N)$, with trivial nebentypus and weight $0$.
For each of these groups, he related Eisenstein series attached to characters to Eisenstein series attached to cusps on $\Gamma(N)$.

\section{Fourier expansion and consequences}
\subsection{The Fourier expansion}
For $k=0$, Huxley \cite{Huxley} stated (without proof) the Fourier expansion of $E_{\chi_1, \chi_2}(z,s)$; a proof may be found in \cite[Section 5.6]{KnightlyLi}.  For general $k$, and $\chi_1 \overline{\chi_2}$ primitive, \cite{DFI} have developed the Fourier expansion.  One may also find a Fourier expansion for $E_{\chi_1, \chi_2,k}$ in \cite{DiamondShurman}.
The author was not able to locate a formula for general $k,\chi_1, \chi_2$.

It is convenient to consider the ``completed'' Eisenstein series defined by
\begin{equation*}
 E_{\chi_1, \chi_2}^*(z,s):= \frac{(q_2/\pi)^s}{i^{-k} \tau(\chi_2)} \Gamma(s+\tfrac{k}{2})  L(2s, \chi_1 \chi_2) E_{\chi_1, \chi_2}(z,s),
\end{equation*}
where $\tau(\chi_2)$ denotes the Gauss sum.

\begin{myprop}
\label{prop:FourierExpansion}
Suppose $\chi_i$ is primitive modulo $q_i$, $i=1,2$, and $(\chi_1 \chi_2)(-1) = (-1)^k$.  Then
\begin{equation*}
E_{\chi_1, \chi_2}^*(z,s)
 = 
 e_{\chi_1, \chi_2}^*(y,s)
 +   \sum_{n \neq 0} \frac{\lambda_{\chi_1,\chi_2}(n,s)}{\sqrt{|n|}} e(nx) \frac{\Gamma(s+\frac{k}{2}) }{\Gamma(s+\frac{k}{2} \sgn(n))}  W_{\frac{k}{2} \sgn(n), s-\frac12}(4 \pi |n| y),
\end{equation*}
where
\begin{equation}
\label{eq:EisensteinChiChiConstantTerm}
\begin{split}
 e_{\chi_1, \chi_2}^*(y,s) = & \delta_{q_1=1} q_2^{2s} \frac{\pi^{-s}}{i^{-k} \tau(\chi_2)} \Gamma(s+\tfrac{k}{2})  L(2s, \chi_2) y^s
 \\
& +  \delta_{q_2=1}  q_1^{2-2s} \frac{ \pi^{-(1-s)} }{ i^{-k} \tau(\overline{\chi_1}) } 
 \Gamma(1-s+\tfrac{k}{2})
 L(2-2s, \overline{\chi_1}) y^{1-s},
 \end{split}
\end{equation}
\begin{equation}
\label{eq:lambdachi1chi2Def}
 \lambda_{\chi_1,\chi_2}(n,s) = \chi_2(\sgn(n)) \sum_{ab=|n|} \chi_1(a) \overline{\chi_2}(b) \Big(\frac{b}{a}\Big)^{s-\frac12},
\end{equation}
 and $W_{\alpha, \beta}$ is the Whittaker function. 
In particular, if $k=0$, then
\begin{equation*}
E_{\chi_1, \chi_2}^*(z,s) = 
e_{\chi_1, \chi_2}^*(y,s)+
2  \sqrt{y}  \sum_{n \neq 0} \lambda_{\chi_1,\chi_2}(n,s) e(nx) K_{s-\frac12}(2 \pi |n| y).
\end{equation*}
\end{myprop}
{\bf Convention.}  Here and throughout we take the convention that the principal character modulo $1$ is primitive.


\begin{proof}
We have
\begin{equation*}
 G_{\chi_1, \chi_2}(z,s) = \sum_{n \in \mz} e(nx) c_n(y), \qquad c_n(y) =  \int_0^{1} G_{\chi_1, \chi_2}(x+iy,s) e(-nx) dx,
\end{equation*}
and inserting \eqref{eq:Gchi1chi2Def}, we have
\begin{equation*}
 c_n(y) = \frac12 \sumprime_{c,d \in \mathbb{Z}} \chi_1(c) \chi_2(d) (q_2 y)^s \int_0^1 \frac{e(-nx)}{|cq_2(x+iy) + d|^{2s}} 
 \Big(\frac{|cq_2(x+iy) + d|}{cq_2(x+iy) + d} \Big)^k
 dx.
\end{equation*}
Extracting the term $c=0$, which only occurs for $q_1 = 1$, we obtain
\begin{equation*}
 c_n(y) = \delta_{q_1=1}  \delta_{n=0} (q_2 y)^s L(2s, \chi_2) 
 + b_n(y),
\end{equation*}
where
\begin{equation*}
 b_n(y) = 
 \sum_{c \geq 1} \sum_{d \in \mz} \chi_1(c) \chi_2(d) (q_2 y)^s \int_0^1 \frac{e(-nx)}{|cq_2(x+iy) + d|^{2s}} 
 \Big(\frac{|cq_2(x+iy) + d|}{cq_2(x+iy) + d} \Big)^k
 dx.
\end{equation*}
Changing variables $x \rightarrow x-\frac{d}{cq_2}$, we obtain
\begin{equation*}
 b_n(y) = (q_2 y)^s   \sum_{c \geq 1}   \frac{\chi_1(c)}{(cq_2)^{2s}} 
 \sum_{d \in \mz}
 \chi_2(d) 
 e\Big(\frac{nd}{c q_2}\Big) 
 \int_{\frac{d}{cq_2}}^{\frac{d}{cq_2} + 1} \frac{e(-nx)}{|x+iy|^{2s}} 
 \Big(\frac{|x+iy |}{x+iy } \Big)^k
 dx.
\end{equation*}
Next break up the sum over $d$ into arithmetic progressions modulo $c q_2$, giving
\begin{equation*}
 b_n(y) = (q_2 y)^s   \sum_{c \geq 1}   \frac{\chi_1(c)}{(cq_2)^{2s}} 
 \sum_{r \shortmod{ c q_2}}
 \chi_2(r) 
 e\Big(\frac{nr}{c q_2}\Big)
 \sum_{\ell \in \mz}
 \int_{\frac{r + \ell cq_2}{cq_2}}^{\frac{r + \ell cq_2}{cq_2} + 1} \frac{e(-nx)}{|x+iy|^{2s}} \Big(\frac{|x+iy |}{x+iy } \Big)^k dx.
\end{equation*}
The sum over $\ell$ forms the complete integral over $\mathbb{R}$, and the sum over $r$ satisfies
\begin{equation*}
 \sum_{r \shortmod{ c q_2}}
 \chi_2(r) 
 e\Big(\frac{nr}{c q_2}\Big) = c \delta(c|n) \tau(\chi_2, n/c) = c \delta(c|n) \overline{\chi_2}(n/c) \tau(\chi_2),
\end{equation*}
where the final equation holds for all $n$ provided $\chi_2$ is primitive (see \cite[(3.12)]{IK}).  
The integral may be evaluated with \cite[3.384.9]{GR}, but it takes some care to transform the integral into this template.  We have
\begin{equation*}
 \intR \frac{e(-nx)}{|x+iy|^{2s}}
 \Big(\frac{|x+iy |}{x+iy } \Big)^k 
 dx = 
 y^{1-2s} \intR \frac{e(-nxy)}{|x+i|^{2s-k}} (x+i)^{-k} dx.
 \end{equation*}
 Then we note $(x+i)^{-k} = i^{-k} (1-ix)^{-k}$, $|x+i|^{2s-k} = |1+ix|^{2s-k} = (1+ix)^{s-\frac{k}{2}} (1-ix)^{s-\frac{k}{2}}$, which is valid since $\arg(1 \pm ix) \in (-\pi/2, \pi/2)$.  Thus the integral we want is
 \begin{equation*}
  y^{1-2s} i^{-k} \intR \frac{\exp(-2\pi i nxy) dx}{(1+ix)^{s-\frac{k}{2}} (1-ix)^{s+\frac{k}{2}} } 
  =  y^{1-2s} i^{-k} 
  \frac{ (2 \pi) 2^{-s} (2 \pi |n| y)^{s-1}}{\Gamma(s + \frac{k}{2} \sgn(n))}
  W_{\frac{k}{2} \sgn(n), s-\frac12}(4 \pi |n| y),
 \end{equation*}
valid for $n \neq 0$ and $\text{Re}(s) > \frac12$.  
 This simplifies as
 \begin{equation*}
 \frac{y^{-s} i^{-k} \pi^s  |n|^{s-1}}{\Gamma(s+\frac{k}{2} \sgn(n))}  W_{\frac{k}{2} \sgn(n), s-\frac12}(4 \pi |n| y).
 \end{equation*}
We note the special cases (see \cite[(9.235.2)]{GR}, \cite[(4.21)]{DFI}):
\begin{equation}
\label{eq:WhittakerDef}
 W_{\frac{k}{2} \sgn(n), s-\frac12}(4 \pi |n| y)
 = \begin{cases}
    2 \sqrt{|n| y} K_{s-\frac12}(2 \pi |n| y)  \qquad &k=0, n \neq 0, \\
   (4 \pi n y)^{k/2} \exp(- 2 \pi n y),    \qquad &s= k/2 > 0, \thinspace n \geq 1.
   \end{cases}
\end{equation}
When $n=0$, we have from \cite[(8.381.1)]{GR}
\begin{equation*}
 \intR \frac{1}{|x+iy|^{2s}}
 \Big(\frac{|x+iy |}{x+iy } \Big)^k 
 dx = i^{-k} (2 \pi)  (2y)^{1-2s} \frac{\Gamma(2s-1)}{\Gamma(s+\tfrac{k}{2}) \Gamma(s-\tfrac{k}{2})}.
\end{equation*}

Thus, $b_0(y) = 0$ if $q_2 > 1$, and for $n \neq 0$, we have
\begin{equation*}
 b_n(y) = 
 (q_2 y)^s   \sum_{c \geq 1}   \frac{\chi_1(c)}{(cq_2)^{2s}} 
 c \delta(c|n) \overline{\chi_2}(n/c) \tau(\chi_2)
 \frac{y^{-s} i^{-k} \pi^s |n|^{s-1} }{\Gamma(s + \frac{k}{2} \sgn(n))}  W_{\frac{k}{2} \sgn(n), s-\frac12}(4 \pi |n| y).
\end{equation*}
Simplifying gives the desired formula for the nonzero Fourier coefficients.

For the constant term, we obtain
\begin{equation*}
 c_0(y) = \delta_{q_1=1}  (q_2 y)^s L(2s, \chi_2)
 + \delta_{q_2=1} y^{1-s} L(2s -1, \chi_1)
 i^{-k} \frac{(2 \pi) 2^{1-2s} \Gamma(2s-1)}{\Gamma(s+\tfrac{k}{2}) \Gamma(s-\tfrac{k}{2})}.
\end{equation*}
We now simplify this.  The functional equation of $L(2s-1,\chi_1)$ gives
\begin{equation*}
 L(2s-1, \chi_1) = \frac{\sqrt{q_1} }{i^{-\delta_1} \tau(\overline{\chi_1}) } \Big(\frac{q_1}{\pi}\Big)^{\frac32 - 2s} \frac{\Gamma(1-s+\tfrac{\delta_1}{2})}{\Gamma(s-\tfrac12 +\tfrac{\delta_1}{2})} L(2 -2s , \overline{\chi_1}),
\end{equation*}
where $\delta_1 = 0$ if $\chi_1(-1) = 1$, and $\delta_1 = 1$ if $\chi_1(-1) = -1$.
Also, note
\begin{equation*}
 \Gamma(2s-1) = \pi^{-1/2} 2^{2s-2} \Gamma(s-\tfrac12) \Gamma(s).
\end{equation*}
Therefore,
\begin{equation*}
\begin{split}
 c_0(y) =  
 &\delta_{q_1=1}  q_2^s L(2s, \chi_2) y^s
 \\
 + &\delta_{q_2=1}  q_1^{2-2s} \frac{ \pi^{2s-1} }{i^{-\delta_1} \tau(\overline{\chi_1}) }  \frac{\Gamma(1-s+\tfrac{\delta_1}{2})}{\Gamma(s-\tfrac12 +\tfrac{\delta_1}{2})} 
 i^{-k} \frac{   \Gamma(s-\tfrac12) \Gamma(s)}{\Gamma(s+\tfrac{k}{2}) \Gamma(s-\tfrac{k}{2})}
 L(2-2s, \overline{\chi_1}) y^{1-s}.
\end{split}
\end{equation*}
Consider
\begin{equation*}
 \gamma(s,\delta_1,k) = \frac{ \Gamma(s-\tfrac12) \Gamma(s) \Gamma(1-s+\tfrac{\delta_1}{2})}{\Gamma(s-\tfrac12 +\tfrac{\delta_1}{2}) \Gamma(s-\tfrac{k}{2})\Gamma(1-s+\tfrac{k}{2})}.
\end{equation*}
Note that in our application, $\delta_1 \equiv k \pmod{2}$.
Using standard gamma function identities, we obtain
\begin{equation*}
 \gamma(s,\delta_1,k) = 
 \begin{cases}
             (-1)^{k/2}, \qquad &k \text{ even}, \\
             (-1)^{(k-1)/2}, \qquad &k \text{ odd},
             \end{cases}
\end{equation*}
whence $i^{-k+\delta_1} \gamma(s,\delta_1,k) = 1$.  Therefore, 
\begin{equation*}
\begin{split}
 c_0(y) =  
 \delta_{q_1=1}  q_2^s L(2s, \chi_2) y^s
 +  \delta_{q_2=1}  q_1^{2-2s} \frac{ \pi^{2s-1} }{ \tau(\overline{\chi_1}) } 
 \frac{\Gamma(1-s+\tfrac{k}{2})}{\Gamma(s+\tfrac{k}{2})}
 L(2-2s, \overline{\chi_1}) y^{1-s}.
\end{split}
\end{equation*}
Then using $e_{\chi_1, \chi_2}^*(y,s) = \frac{(q_2/\pi)^s}{i^{-k} \tau(\chi_2)} \Gamma(s+\frac{k}{2}) c_0(y)$ completes the proof.
\end{proof}

  Note that
\begin{equation}
\label{eq:alphasum}
 \sum_{n=1}^{\infty} \frac{\lambda_{\chi_1,\chi_2}(n,s)}{n^u} = L(u+s-\tfrac12 , \chi_1) L(u + \tfrac12-s, \overline{\chi_2}).
\end{equation}
\subsection{Functional equations}
The Fourier coefficient satisfies the functional equation
\begin{equation}
\label{eq:FourierCoefficientFunctionalEquation}
 \lambda_{\chi_1,\chi_2}(n,1-s) = (\chi_1 \chi_2)(\sgn(n)) \lambda_{\overline{\chi_2},\overline{\chi_1}}(n,s).
\end{equation}
As for the Eisenstein series itself, we have
\begin{myprop}
 Suppose $\chi_i$ is primitive modulo $q_i$, $i=1,2$, and $(\chi_1 \chi_2)(-1) = (-1)^k$.  Then $E_{\chi_1, \chi_2}^*(z,s)$ extends to a meromorphic function for all $s \in \mathbb{C}$.  Moreover, it satisfies the functional equation
 \begin{equation}
 \label{eq:Echi1chi2FunctionalEquation}
E^*_{\chi_1,\chi_2}(z,s) 
=  E^*_{\overline{\chi_2}, \overline{\chi_1}}(z,1-s).
 \end{equation}
 If $k \geq 0$ and $q_1 q_2 > 1$, then $E_{\chi_1, \chi_2}^*(z,s)$ is analytic for all $s \in \mc$.
\end{myprop}
Remark. For $k<0$, the multiplication by $\Gamma(s+\frac{k}{2})$ produces some poles of $E_{\chi_1,\chi_2}^*(z,s)$.  If desired, one could form the completed Eisenstein series by multiplication by $\Gamma(s-\frac{k}{2})$ instead of $\Gamma(s+\frac{k}{2})$, leading to a slightly different functional equation following from \eqref{eq:GammaFunctionalEquation} below.

\begin{proof}
 First consider the case with $q_1, q_2 > 1$.
 From the Fourier expansion, we have
 \begin{equation*}
 \begin{split}
  E^*_{\chi_1,\chi_2}(z,1-s)
  &= 
  \sum_{n=1}^{\infty} 
   \frac{\lambda_{\chi_1,\chi_2}(n,1-s)}{\sqrt{n}} e(nx)  W_{\frac{k}{2} , s-\frac12}(4 \pi |n| y)
  \\
  &+ 
   \sum_{n=1}^{\infty}
   \frac{\lambda_{\chi_1,\chi_2}(-n,1-s)}{\sqrt{n}} e(-nx) \frac{\Gamma(1-s+\frac{k}{2}) }{\Gamma(1-s-\frac{k}{2})}  W_{-\frac{k}{2} , s-\frac12}(4 \pi |n| y).
   \end{split}
 \end{equation*}
 Since the Whittaker function has exponential decay, uniformly for $s$ on compact sets, we see that $E_{\chi_1,\chi_2}^*(z,1-s)$ extends to a meromorphic function with possible poles at $s=\tfrac{k}{2} + \ell$ with $\ell$ a nonnegative integer
 
Using \eqref{eq:FourierCoefficientFunctionalEquation}, $\chi_1(-1) \chi_2(-1) = (-1)^k$, and
\begin{equation}
\label{eq:GammaFunctionalEquation}
 \frac{\Gamma(1-s+\frac{k}{2}) }{\Gamma(1-s-\frac{k}{2})}  = (-1)^k \frac{\Gamma(s+\frac{k}{2}) }{\Gamma(s-\frac{k}{2})},
\end{equation}
we obtain the functional equation.  Moreover, the poles at $1-s$ with  $s=\frac{k}{2}+\ell$ are removable, provided $k \geq 0$.

Next consider the case with $q_1 = 1$ or $q_2 = 1$.  The non-constant terms in the Fourier expansion have the same analytic properties as in the case with $q_1, q_2 > 1$, so it suffices 
to examine the constant term $e_{\chi_1,\chi_2}^*(y,s)$.  
By inspection of \eqref{eq:EisensteinChiChiConstantTerm}, it satisfies the functional equation
\begin{equation*}
 e_{\chi_1, \chi_2}^*(y,1-s) = e_{\overline{\chi_2}, \overline{\chi_1}}^*(y,s).
\end{equation*}
The coefficient of $y^s$ is analytic except possibly for poles at $s = -\frac{k}{2} - \ell$, with $\ell$ a nonnegative integer.  If $k \geq 0$ then these poles are cancelled by the trivial zeros of $L(2s, \chi_2)$.  Similarly, the coefficient of $y^{1-s}$ is analytic except for possible poles at $s= \frac{k}{2} + 1 + \ell$, with $\ell$ a nonnegative integer.  These points occur at $2-2s = -k - 2 \ell$ and are also cancelled by trivial zeros of $L(2-2s, \overline{\chi_1})$, provided $k \geq 0$.
%
%
%
\end{proof}

 We are also interested in the Mellin transform of $E_{\chi_1,\chi_2}^*(iy,s)$ and its functional equation.  A motivation for this explicit calcluation comes from the evaluation of restriction norms of automorphic forms, as in \cite{YoungQUE}.
  See \cite[Section 8]{DFI} for this calculation with cusp forms.
 The reader only interested in the change of basis formulas may wish to skip these calculations.
 
 Assume that $\chi_1$ and $\chi_2$ are primitive with $q_1, q_2 >1$, so the constant term vanishes.  We have
 \begin{multline}
\label{eq:Mellin}
 \int_0^{\infty} E^*_{\chi_1,\chi_2}(iy,s) (q_1 q_2)^{u/2} y^u \frac{dy}{y} 
\\ 
 =  (q_1 q_2)^{u/2}  \sum_{n=1}^{\infty} 
 \frac{\lambda_{\chi_1, \chi_2}(n,s)}{\sqrt{n}} 
 \int_0^{\infty} 
 \Big(W_{\frac{k}{2}, s-\frac12}(4 \pi n y) + \varepsilon_2 \frac{\Gamma(s+\frac{k}{2})}{\Gamma(s-\frac{k}{2})} W_{-\frac{k}{2}, s-\frac12}(4 \pi n y)
 \Big) y^{u} \frac{dy}{y},
\end{multline}
with $\varepsilon_i = \chi_i(-1)$, for $i=1,2$.  Changing variables $y \rightarrow \frac{y}{\pi n}$, and evaluating the Dirichlet series using \eqref{eq:alphasum}, we have
\begin{multline}
\label{eq:EisensteinCompletedMellinTransform}
 \int_0^{\infty} E^*_{\chi_1,\chi_2}(iy,s) (q_1 q_2)^{u/2} y^u \frac{dy}{y}
 \\
 = 
 \frac{1}{\sqrt{\pi}} \Big(\frac{q_1}{\pi}\Big)^{\frac{u}{2}} 
 \Big(\frac{q_2}{\pi}\Big)^{\frac{u}{2}} 
 L(\tfrac12 + u+s-\tfrac12 , \chi_1) L(\tfrac12 + u + \tfrac12-s, \overline{\chi_2})
 \Phi_k^{\varepsilon_2}(u+\tfrac12,s-\tfrac12),
\end{multline}
where as in \cite[(8.25)]{DFI}
\begin{equation}
\label{eq:PhikDef}
 \Phi_k^{\varepsilon}(\alpha,\beta)
 = \sqrt{\pi} \int_0^{\infty} 
 \Big(W_{\frac{k}{2}, \beta}(4  y) + \varepsilon \frac{\Gamma(\beta+\frac{1+k}{2})}{\Gamma(\beta+\frac{1-k}{2})} W_{-\frac{k}{2}, \beta}(4  y)
 \Big) y^{\alpha-\frac12} \frac{dy}{y}.
\end{equation}
Actually, our definition of $\Phi_k^{\varepsilon}$ differs from \cite{DFI} in that we have not divided by $4$ in the right hand side of \eqref{eq:PhikDef}.
This integral is evaluated in \cite[Lemma 8.2]{DFI}, for $k \geq 0$, but the value stated there is not quite correct.  The correct formula is
\begin{equation}
\label{eq:PhikCorrectedFormula}
 \Phi_k^{\varepsilon}(\alpha,\beta)
 =  p_k^{\varepsilon}(\alpha,\beta)
 \Gamma\Big(\frac{\alpha+\beta + \frac{1- \varepsilon (-1)^k}{2}}{2}\Big)
 \Gamma\Big(\frac{\alpha-\beta + \frac{1-\varepsilon}{2}}{2}\Big),
\end{equation}
 and $p_k^{\varepsilon}(\alpha,\beta)$ is a certain polynomial in $\alpha$, defined recursively.  The formula of \cite{DFI} is incorrect in a few ways.  One simple mistake is that it is off by a factor of
 $4$, explaining our change in definition, and the more important error boils down to interchanging $\frac{1-\varepsilon}{2}$ and $\frac{1-\varepsilon(-1)^k}{2}$ in the gamma factors, which arises from a typo early in the calculation of \cite{DFI}.  A  side effect of this typo is that the formula for $p_k^{\varepsilon}(\alpha,\beta)$ needs correction.
 We have devoted Section \ref{section:WhittakerMellinTransform} to correcting the evaluation of $\Phi_k^{\varepsilon}$.
 
  For $k=0,1$, the polynomial is given by
 \begin{equation*}
  p_0^{\varepsilon}(\alpha,\beta) = \tfrac{1+\varepsilon}{2},
  \qquad
  p_1^{\varepsilon}(\alpha,\beta) = 1.
 \end{equation*}
 Note that if $\varepsilon = \varepsilon_2$, then $\varepsilon_2 (-1)^k = \varepsilon_1$, and so in our desired application the gamma factors match those of the Dirichlet $L$-functions appearing in \eqref{eq:EisensteinCompletedMellinTransform} (only after the correction \eqref{eq:PhikCorrectedFormula}), keeping in mind how the gamma factor depends on the parity of the character.
 
Define $\Lambda(s,\chi) = (q/\pi)^{s/2} \Gamma(\frac{s+\delta}{2}) L(s,\chi)$ for a primitive Dirichlet character modulo $q$, with $\delta = \frac{1-\chi(-1)}{2}$.  Then, for $k=0,1$, we have
\begin{equation*}
 \int_0^{\infty} E^*_{\chi_1,\chi_2}(iy,s) (q_1 q_2)^{u/2} y^u \frac{dy}{y}
 = \delta_3  
 q_1^{\frac{-s}{2}} 
 q_2^{\frac{s-1}{2}} 
 \Lambda(\tfrac12 + u+s-\tfrac12 , \chi_1) \Lambda(\tfrac12 + u + \tfrac12-s, \overline{\chi_2}),
\end{equation*}
where $\delta_3 = 1$ 
unless $k=0$ and $\varepsilon_1 = \varepsilon_2 = -1$, in which case $\delta_3 = 0$.

From the functional equation $\Lambda(s,\chi) = \epsilon(\chi) \Lambda(1-s,\overline{\chi})$, we see that \eqref{eq:Mellin} satisfies
\begin{equation*}
 \int_0^{\infty} E_{\chi_1, \chi_2}^*(iy,s) (q_1 q_2)^{u/2} y^u \frac{dy}{y} = \epsilon(\chi_1) \epsilon(\overline{\chi_2}) 
 q_1^{\frac12-s} q_2^{s-\frac12}
 \int_0^{\infty} E_{\chi_2, \chi_1}^*(iy,s) (q_1 q_2)^{-u/2} y^{-u} \frac{dy}{y}.
\end{equation*}
On the other hand, if we apply the Fricke involution to $E_{\chi_1,\chi_2}^*$, which maps $y$ to $\frac{1}{q_1 q_2 y}$, then we get that
\begin{equation*}
 \int_0^{\infty} E^*_{\chi_1,\chi_2}(iy,s) (q_1 q_2)^{u/2} y^u \frac{dy}{y} = \int_0^{\infty}  E^*_{\chi_1,\chi_2} \Big(\frac{i}{q_1 q_2 y},s\Big) (q_1 q_2)^{-u/2} y^{-u} \frac{dy}{y}.
\end{equation*}
Hence for $\delta_3 = 1$, we have
\begin{equation}
\label{eq:FrickeFormula1}
E^*_{\chi_1,\chi_2} \Big(\frac{i}{q_1 q_2 y},s\Big) = \epsilon(\chi_1) \epsilon(\overline{\chi_2}) q_1^{\frac12-s} q_2^{s-\frac12}   E_{\chi_2,\chi_1}^*(iy,s).
\end{equation}
In Section \ref{section:AtkinLehnerEigenvalues}, we will explicitly derive the action of all the Atkin-Lehner operators for $\chi_i$ even or odd, which will generalize \eqref{eq:FrickeFormula1}.  This provides a pleasant consistency check.
 

\subsection{Hecke operators}
One simple consequence of the explicit calculation of the Fourier expansion, particularly the Euler product formula implied by \eqref{eq:alphasum}, is that this shows that $E_{\chi_1, \chi_2}(z,s)$ is an eigenfunction of all the Hecke operators, including $p|q_1 q_2$.   
We have that
\begin{equation}
T_n E_{\chi_1,\chi_2}(z,1/2+it) = 
\lambda_{\chi_1,\chi_2}(n,1/2+it)
E_{\chi_1,\chi_2}(z,1/2+it).
\end{equation}
In particular, if $p|(q_1, q_2)$ then $T_p E_{\chi_1, \chi_2} = 0$.

\subsection{Holomorphic forms}
It would be negligent not to extract information on the holomorphic Eisenstein series defined by \eqref{eq:EholomorphicDef}.  
Formally specializing $s=k/2$, we obtain
\begin{equation*}
  E_{\chi_1, \chi_2}^*(z,k/2) = 
 e_{\chi_1,\chi_2}^*(y,k/2) +  
 \sum_{n=1}^{\infty} 
 \frac{\lambda_{\chi_1,\chi_2}(n,k/2)}{\sqrt{n}} e(nx)  (4 \pi n y)^{k/2} \exp(- 2 \pi n y),
\end{equation*}
using \eqref{eq:WhittakerDef} to simplify the Whittaker function. 
In the constant term, we have $\delta_{q_2=1} L(2-k, \overline{\chi_1}) = 0$ for $k \geq 2$.  Therefore, for $k \geq 2$ we have
\begin{multline*}
\frac{(\frac{q_2}{2 \pi})^{k}  \Gamma(k) L(k, \chi_1 \chi_2)}{i^{-k} \tau(\chi_2)} E_{\chi_1, \chi_2, k}(z) = 
\delta_{q_1=1} \frac{(\frac{q_2}{2 \pi})^{k}  \Gamma(k) L(k, \chi_1 \chi_2)}{i^{-k} \tau(\chi_2)}
\\
+ 
\sum_{n=1}^{\infty} \lambda_{\chi_1, \chi_2}(n,k/2) n^{\frac{k-1}{2}} e(nz).
\end{multline*}
Compare with \cite[Theorem 4.5.1]{DiamondShurman}, and note 
\begin{equation*}
 \lambda_{\chi_1,\chi_2}(n,k/2) n^{\frac{k-1}{2}} = \sum_{ab=n} \chi_1(a) \overline{\chi_2}(b) b^{k-1}.
\end{equation*}
%
%
Since the original definition of $E_{\chi_1, \chi_2}(z,s)$ converges absolutely for $\text{Re}(s) > 1$, then certainly we may set $s = k/2$ for $k \geq 3$.  When $k=2$, then since the completed Eisenstein series is entire in $s$, we may also set $s=k/2 = 1$, except when $q_1 = q_2=1$ in which case the level $1$ Eisenstein series (with $\chi_1 = \chi_2 = 1$) has a pole at $s=1$.  See \cite[Section 4.6]{DiamondShurman} for a description of the linear space of weight $2$ Eisenstein series.

We may also examine $k=1$.  Suppose $q_1, q_2 > 1$ for ease of discussion; then we may set $s=k/2 = 1/2$.
One interesting feature here is that for weight $1$, $\lambda_{\chi_1, \chi_2}(n,1/2) = \lambda_{\overline{\chi_2}, \overline{\chi_1}}(n,1/2)$ for $n \geq 1$ (see \eqref{eq:FourierCoefficientFunctionalEquation}) showing that $E_{\chi_1, \chi_2, 1}(z)$ is a scalar multiple of $E_{\overline{\chi_2}, \overline{\chi_1},1}(z)$.  This corresponds, roughly, to the fact that the dimension of the space of holomorphic Eisenstein series of weight $1$ is about half that of the corresponding space of odd weight $k \geq 3$.  See \cite[Section 4.8]{DiamondShurman} for precise statements.

\section{Preliminary formulas}
Now we embark on proving the change of basis formulas.
\subsection{Primitive and non-primitive}
\begin{mylemma}
\label{lemma:EisensteinPrimitivevsNonPrimitive}
Suppose that $\chi_i$ has modulus $q_i$, and is induced by the primitive character $\chi_i^*$ of modulus $q_i^*$.  Then
\begin{equation*}
E_{\chi_1, \chi_2}(z,s) = \frac{L(2s, \chi_1^* \chi_2^*)}{L(2s, \chi_1 \chi_2)}
\sum_{a|q_1} \sum_{b|q_2} \frac{\mu(a) \chi_1^*(a) \mu(b) \chi_2^*(b)}{(ab)^s}
E_{\chi_1^*, \chi_2^*}\Big(\frac{a q_2 }{b q_2^* } z, s\Big).
\end{equation*}
\end{mylemma}
Remarks.  Within the sum, we may restrict to $(b, q_2^*) =1$, which implies $b q_2^* | q_2$, and so $\frac{a q_2 }{b q_2^* }$ is an integer.  Similarly, we may assume $(a, q_1^*) = 1$ which gives that $\frac{a q_2 }{b q_2^* }$ is a divisor of $\frac{q_1 q_2}{q_1^* q_2^*}$.

\begin{proof}
The desired formula is equivalent to
\begin{equation*}
G_{\chi_1, \chi_2}(z,s) = 
\sum_{a|q_1} \sum_{b|q_2} \frac{\mu(a) \chi_1^*(a) \mu(b) \chi_2^*(b)}{(ab)^s}
G_{\chi_1^*, \chi_2^*}\Big(\frac{aq_2  }{bq_2^*  } z, s\Big).
\end{equation*}
By definition,
\begin{equation*}
G_{\chi_1, \chi_2}(z,s) = 
\frac12 \sumprime_{\substack{c,d \in \mathbb{Z} \\ (c, q_1) = 1 \\ (d, q_2) = 1 }} \frac{(q_2 y)^s \chi_1^*(c) \chi_2^*(d)}{|cq_2z + d|^{2s}} \Big(\frac{|cq_2 z + d|}{cq_2 z + d}\Big)^{k} .
\end{equation*}
By M\"obius inversion, we deduce
\begin{equation*}
G_{\chi_1, \chi_2}(z,s) = 
\sum_{a|q_1} \sum_{b|q_2} \mu(a) \mu(b) \chi_1^*(a) \chi_2^*(b)
\frac12 \sumprime_{\substack{c,d \in \mathbb{Z} }} \frac{(q_2 y)^s \chi_1^*(c) \chi_2^*(d)}{|a c  q_2 z + bd|^{2s}} \Big(\frac{|acq_2 z + bd|}{acq_2 z + bd}\Big)^{k}.
\end{equation*}
For the inner sum above, we have
\begin{equation*}
\frac12 \sumprime_{\substack{c,d \in \mathbb{Z} }} \frac{(q_2 y)^s \chi_1^*(c) \chi_2^*(d)}{|a c  q_2 z + bd|^{2s}} 
\Big(\frac{|acq_2 z + bd|}{acq_2 z + bd}\Big)^{k}
=
\frac{1}{(ab)^{s}}  \frac12 \sumprime_{\substack{c,d \in \mathbb{Z} }} \frac{( q_2^*\frac{aq_2}{b q_2^*} y)^s \chi_1^*(c) \chi_2^*(d)}{| c q_2^* \frac{aq_2}{b q_2^*}  z + d|^{2s}}
\Big(\frac{| c q_2^* \frac{aq_2}{b q_2^*}  z + d|}{ c q_2^* \frac{aq_2}{b q_2^*}  z + d}\Big)^k,
\end{equation*}
which is none other than $(ab)^{-s} G_{\chi_1^*, \chi_2^*}(\frac{a q_2}{b q_2^*} z, s)$.  
\end{proof}

\subsection{Notation and results from \cite{KY}}
\label{section:KY}
As shown in \cite[Proposition 3.1]{KY} (see also \cite[Section 3.8]{DiamondShurman}), a complete set of inequivalent cusps for $\Gamma_0(N)$ is given by $\frac{1}{w} = \frac{1}{uf}$ where $f | N$ and $u$ runs modulo $(f, N/f)$, coprime to the modulus, after choosing a representative coprime to $N$ (such a choice can always be achieved).  
With this choice of representative, then $\frac{1}{uf} \sim \frac{u}{f}$ (equivalent in $\Gamma_0(N)$), and so by \eqref{eq:EaEquivalentCuspsRelation}, $E_{u/f} =  
\overline{\psi}(\gamma)
E_{1/w}$, where $\gamma(\frac{1}{uf}) = \frac{u}{f}$.  One may easily check that any $\gamma$ satisfying this equation has lower-right entry $d$ congruent to $\overline{u}$ modulo both $f$ and $N'=N/f$.  Thus $d \equiv \overline{u} \pmod{[f,N']}$.  
In Lemma \ref{lemma:singularcusps} below, we show that $u/f$ is singular for $\psi$ iff $\psi$ has period dividing $[f,N']$, and therefore $\overline{\psi}(\gamma) = \psi(u)$.  That is,
\begin{equation}
\label{eq:Euoverfversus1overw}
 E_{u/f}(z,s,\psi) = \psi(u) E_{1/w}(z,s,\psi).
\end{equation}
We will generally work with cusps of the form $\frac{1}{uf}$, but  one may convert to $\frac{u}{f}$ using \eqref{eq:Euoverfversus1overw}.


\begin{myprop}[\cite{KY}, Proposition 3.3]\label{prop:stabilizerScaling}
            Let $\c= 1/w$ be a cusp of $\Gamma = \Gamma_0(N)$, and set
            \begin{equation}
            \label{eq:Nandwformulas}
              N = (N,w) N'   \qquad w = (N,w) w', 
              \qquad
              N' = (N',w) N''.  
            \end{equation}
            The stabilizer of $1/w$ is given as 
            \begin{equation}
            \label{eq:stabilizerFormula}
                \Gamma_{1/w} = \left\{ \pm \tau_{1/w}^t  : t \in \Z \right\}, \quad \text{where} \quad \tau_{1/w}^t = \bpm 1 - wN'' t  & N'' t \\ -w^2 N'' t & 1+ wN'' t \epm,
            \end{equation}
            and one may choose the scaling matrix as
           \begin{equation}\label{eq:anyCuspScaling}
               \sigma_{1/w} = \bpm 1&0\\ w&1 \epm  \bpm \sqrt{N''}  &0\\ 0 &1/\sqrt{N''} \epm .
           \end{equation}
        \end{myprop}
 Remark.  With this choice of scaling matrix, we have
 \begin{equation}
 \label{eq:scalingMatrixGeneratorComparedtoGammaInfinityGenerator}
  \sigma_{1/w}^{-1} \tau_{1/w} \sigma_{1/w} = \bpm 1 & 1 \\ & 1 \epm,
 \end{equation}
which is important in the context of checking if $1/w$ is singular for a Dirichlet character $\psi$ (recall the discussion in Section \ref{section:EisensteinCusps}).  One should also observe that $N | w^2 N''$ to see that $\tau_{1/w} \in \Gamma_0(N)$.

 We next quote a double coset calculation from \cite{KY}, in the special case $\a = \infty$, in which case the notation from \cite{KY} specializes with $r=N$, $s=1$:
\begin{mylemma}[\cite{KY}, Lemma 3.5]
	\label{lemma:DoubleCosetFormula}
		Let $\c = 1/w$ be any cusp of $\Gamma=\Gamma_0(N)$ and $\a = 1/N \sim \infty$. 
		Let the scaling matrix $\sigma_{1/w}$ be as in 
		\eqref{eq:anyCuspScaling}, and take $\sigma_{\infty} = I$. Then 
		\begin{equation}
		\label{eq:doublecoset}
			\sigma_{1/w} \inv \Gamma \sigma_\infty
			= \left\{\bpm \frac{A}{\sqrt{N''}} &\frac{B}{\sqrt{N''}} \\ C \sqrt{N''} &D \sqrt{N''} \epm: 
			\bpm A & B \\ C & D \epm \in \SL_2(\Z),\,\,
			C \equiv - wA \mymod{N} 
			\right\}.
		\end{equation}
	\end{mylemma}

It is convenient to translate the notation a bit.  With $w =uf$ as above, we have 
\begin{equation*}
 N' = \frac{N}{f}, \qquad N'' = \frac{N'}{(f,N')}, \qquad w'=u.
\end{equation*}

\subsection{Singular cusps} 
 \begin{mylemma}
 \label{lemma:singularcusps}
 Let $\psi$ be a Dirichlet character modulo $N$, and let $\a = \frac{1}{uf}$ with $f|N$, and $(u,N) = 1$.  Then $\a$ is singular for $\psi$ if and only if $\psi$ is periodic modulo $\frac{N}{(f,N/f)}$, equivalently, the primitive character inducing $\psi$ has modulus dividing $\frac{N}{(f,N/f)}$.
\end{mylemma}
Remarks.  In case $\psi$ is primitive modulo $N$, then the  singular cusps for $\psi$ are the Atkin-Lehner cusps $1/f$ with $(f,N/f) = 1$.
Also, observe $\frac{N}{(f,N/f)} = [f,N/f]$, and so in particular, $N$ and $\frac{N}{(f,N/f)}$ share the same prime factors.

\begin{proof}
  By Proposition \ref{prop:stabilizerScaling} and \eqref{eq:scalingMatrixGeneratorComparedtoGammaInfinityGenerator},
the cusp $\frac{1}{w} = \frac{1}{uf}$ is singular for $\psi$ iff $\psi(1+wN'' t) = 1$ for all $t \in \mathbb{Z}$.  Since $w N'' = \frac{N}{(f,N/f)} u$, the condition that $1/w$ is singular for $\psi$ is seen to be equivalent to $\psi(1+ \frac{N}{(f,N/f)} t) = 1$ for all $t \in \mathbb{Z}$.  

Let $\chi$ be a Dirichlet character modulo $N$, and suppose $d|N$.  It is an elementary exercise to show that $\chi$ is induced by a character of modulus $d$ if and only if $\chi(1+dk) = 1$ for all $k \in \mz$ such that $(1+dk,N) = 1$.  
This exercise completes the proof.
 %
\end{proof}

\begin{mycoro}
Suppose $\psi$ is a Dirichlet character modulo $N$, induced by a primitive character $\psi^*$ of conductor $N^* |N$.  Then the number of singular cusps for $\psi$ equals
\begin{equation*}
 \sum_{\substack{f|N \\ (f,N/f) | \frac{N}{N^*}}} \varphi((f,N/f)).
\end{equation*}
\end{mycoro}

\section{Decomposition of $E_{\a}$}
Our first main result decomposes $E_{\frac{1}{uf}}$ in terms of $E_{\chi_1,\chi_2}$'s.  
 \begin{mytheo}
 \label{thm:EcuspInTermsofEchichi}
  Let notation be as in Section \ref{section:KY}.  Then 
\begin{multline}
\label{eq:EcuspInTermsofEchichi}
 E_{\frac{1}{uf}}(z,s,\psi) = \frac{1}{(fN'')^s}
 \frac{1}{\varphi((f,N/f))} 
 \sum_{q_1 | \frac{N}{f}} \sum_{q_2 | f} 
 \sumstar_{\substack{\chi_1 \shortmod{q_1} \\ \chi_2 \shortmod{q_2} \\ \chi_1 \overline{\chi_2} \sim \psi}} 
  \overline{\chi_1}(-u)
  \frac{L(2s,\chi_1 \chi_2)}{L(2s,\chi_1 \chi_2 \chi_{0,N})}
 \\
\sum_{\substack{a|f \\ (a, q_2)=1}} \sum_{\substack{b|\frac{N}{f} \\ (b, q_1) = 1}} \frac{\mu(a) \mu(b) \chi_1(b)  \chi_2(a)}{(ab)^s}
E_{\chi_1, \chi_2}\Big(\frac{b f}{a q_2} z, s\Big),
\end{multline}
where the sum is over primitive characters $\chi_i$ modulo $q_i$, and $\chi_1 \overline{\chi_2} \sim \psi$ means that both sides are induced by the same primitive character.
 \end{mytheo}
Booker, Lee, and Str\"ombersson (personal communication) have independently proved Theorem \ref{thm:EcuspInTermsofEchichi} as well as the inversion formula in Theorem \ref{thm:EchichiInTermsofEa}.  They use these formulas, as well as the functional equation of $E_{\chi_1, \chi_2}$, to work out the scattering matrix for $\Gamma_0(N)$ with arbitrary nebentypus.  Previously, Huxley \cite{Huxley} considered the trivial nebentypus case.

Remarks.  Suppose that $\psi$ is primitive of conductor $N$, so that the cusp $\frac{1}{uf}$ is singular iff $(f,N/f) = 1$, and so we may take $u=1$.  Then \eqref{eq:EcuspInTermsofEchichi} simplifies as
\begin{equation}
 E_{1/f}(z,s,\psi) = \frac{\chi_1(-1)}{N^s} E_{\chi_1, \chi_2}(z,s),
\end{equation}
where $\chi_1$ is modulo $N/f$ and $\chi_2$ is modulo $f$, and $\chi_1 \overline{\chi_2} = \psi$.  This type of identity is implicit in \cite{DFI}, where the authors explicitly evaluated many properties of the Eisenstein series when the nebentypus is primitive.
In another special case where $N$ is square-free and $\psi$ is principal, 
then \eqref{eq:EcuspInTermsofEchichi} reduces to \cite[(3.25)]{ConreyIwaniec}.

 
 Theorem \ref{thm:EcuspInTermsofEchichi} shows in an explicit form that the space of Eisenstein series is spanned by $E_{\chi_1,\chi_2}(Bz,s)$ with $q_1 q_2 B | N$, and $\chi_1 \overline{\chi_2} \sim \psi$, as expected from the discussion in Section \ref{section:newformdiscussion}.

The proof of Theorem \ref{thm:EcuspInTermsofEchichi} is long, so we break the proof into more managable pieces.  We begin with some notation.  From Lemma \ref{lemma:singularcusps}, the cusp $\frac{1}{uf}$ is singular iff $\psi$ is periodic modulo $[f,N']$.  There exist integers $f_0 | f$ and $N_0'|N'$ so that $[f,N'] = f_0 N_0'$ and $(f_0 ,N_0') = 1$.  The choices of $f_0$ and $N_0'$ may not be unique in case there is a prime power exactly dividing both $f$ and $N'$.  Then we may write $\psi = \psi^{(f_0)} \psi^{(N_0')}$ according to this factorization.
We remark that $f_0$ and $N_0'$ are useful within the proof of Theorem \ref{thm:EcuspInTermsofEchichi}, yet they are not present in the final formula \eqref{eq:EcuspInTermsofEchichi}.

With this notation in place, it is helpful for later to record that if $u \equiv u' \pmod{(f,N')}$ (both coprime to $N$), then
$E_{\frac{1}{uf}}(z,s,\psi) = \psi^{(N_0')}(u' \overline{u}) E_{\frac{1}{u'f}}(z,s,\psi)$, following from \eqref{eq:EaEquivalentCuspsRelation} and a calculation of the lower-right entry of a matrix $\gamma$ such that $\gamma \frac{1}{uf} = \frac{1}{u'f}$.  Put another way, $\psi^{(N_0')}(u) E_{\frac{1}{uf}}(z,s,\psi)$ is a well-defined function of $u \pmod{(f,N')}$.

\begin{mylemma}
\label{lemma:EcuspIntermediateFormula}
 With notation as above, we have
 \begin{equation}
 \label{eq:EufInsideLemma}
 E_{\frac{1}{uf}}(z,s, \psi) 
 = 
 \delta_{f=N} y^s + \frac{y^s}{(N'')^s}
 \sum_{\substack{(D, fC') = 1 \\ C' > 0, \thinspace (C',N') = 1 \\ D \equiv - \overline{C'} u \shortmod{(f,N')} }}
 \frac{ \psi^{(N_0')}(-\overline{u} C') \psi^{(f_0)}(\overline{D}) }{|C'f z + D|^{2s}} \Big(\frac{|C'fz + D|}{C'fz + D}\Big)^k.
\end{equation}
\end{mylemma}
\begin{proof}
 From \eqref{eq:Eadef}, and changing variables, we have
\begin{multline*}
 E_{1/w}(z,s, \psi) = \sum_{\gamma \in \Gamma_{1/w} \backslash \Gamma} \overline{\psi}(\gamma) 
 j(\sigma_{1/w}^{-1} \gamma, z)^{-k}
 \text{Im}(\sigma_{1/w}^{-1} \gamma z)^s 
 \\
 = \sum_{\tau \in \Gamma_{\infty} \backslash \sigma_{1/w}^{-1} \Gamma} \overline{\psi}(\sigma_{1/w} \tau) 
 j(\tau, z)^{-k}
 \text{Im}(\tau z)^s.
\end{multline*}
For the evaluation of $\overline{\psi}(\sigma_{1/w} \tau)$, note that
\begin{equation*}
 \bpm 1&0\\ w&1 \epm  \bpm \sqrt{N''}  &0\\ 0 &1/\sqrt{N''} \epm \bpm \frac{A}{\sqrt{N''}} &\frac{B}{\sqrt{N''}} \\ C \sqrt{N''} &D \sqrt{N''} \epm 
 =
 \begin{pmatrix}
  A & B \\
  C+Aw & D+ Bw
 \end{pmatrix}
.
\end{equation*}
As a consistency check, observe that if we translate $\tau$ on the left by $(\begin{smallmatrix} 1 & n \\ 0 & 1 \end{smallmatrix})$ then that replaces $B$ by $B + D n N''$, and so the lower-right entry changes from $D+Bw$ to $D+Bw + DwnN''$.  Since $wN'' = u \frac{N}{(f,N')}$, and $\psi$ is assumed to be periodic modulo $\frac{N}{(f,N')}$, this shows that $\psi$ is well-defined under such translations.

Next we need to work out representatives for $\Gamma_{\infty} \backslash \sigma_{1/w}^{-1} \Gamma$, in terms of the lower row of matrices occuring in \eqref{eq:doublecoset}.
In the case of the identity coset with $C=0, D=1$, we obtain $\psi(D+Bw) = \psi(1+Bw) = \psi(1)=1$, since this coset occurs only when $f=N$, i.e., $\frac{1}{uf} \sim \infty$.  This leads to the term $\delta_{f=N} y^s$.

From now on, consider the non-identity cosets.  
Note that the action of $\Gamma_{\infty}$ does not affect the congruence linking $A$ to $C$.
Consider the conditions
\begin{equation}
\label{eq:CDconditions}
C > 0, \quad (C,D) = 1, \quad C=fC',   \quad (C',N') =1, \quad \text{and} \quad D \equiv - u \overline{C'} \pmod{(f,N')}.
\end{equation}
We claim that
\begin{equation}
\label{eq:doublecosetDecomposition}
 \Gamma_{\infty} \backslash \sigma_{1/w}^{-1} \Gamma = 
 \delta_{f=N} \Gamma_{\infty} \cup 
 \left\{ \begin{pmatrix} * & * \\ C \sqrt{N''} & D \sqrt{N''} \end{pmatrix} : \eqref{eq:CDconditions} \text{ holds} \right\},
\end{equation}
as a disjoint union.
Moreover, the value of $\overline{\psi}(\sigma_{1/w} \tau)$ is determined by the conditions \eqref{eq:CDconditions} (we will derive a formula for it within the proof).


\emph{Proof of claim.}
First assume that \eqref{eq:CDconditions} holds.
From $(C,D) = 1$, there exists integers $A_0, B_0$ so that $A_0D - B_0 C = 1$.  Then $A_0 \equiv \overline{D} \pmod{C}$, and so from the congruence on $D$ in \eqref{eq:CDconditions}, and the fact that $f|C$, we have $A_0 \equiv - C' \overline{u} \pmod{(f,N')}$.  Let $x,y \in \mathbb{Z}$ be such that $A_0 =- \overline{u} C' + fx + N'y$.  We next want to find $n \in \mz$ so that $A = A_0 + nC$ satisfies $C \equiv - w A \pmod{N}$, which in turn is equivalent to $A \equiv -\overline{u} C' \pmod{N'}$.  For this, we have $A = -\overline{u} C' + fx + N'y + n C'f \equiv -\overline{u} C' + f(x+nC') \pmod{N'}$, so choosing $n \equiv - x \overline{C'} \pmod{N'}$ finishes the job.

Next we show the conditions \eqref{eq:CDconditions} follow from the conditions on the right hand side of \eqref{eq:doublecoset}.  This can be seen as follows.  The determinant equation obviously implies $(C,D) = 1$, and using the congruence we have $1 = AD - BC \equiv A(D+Bw) \pmod{N}$, whence $(A,N) = 1$, and so $(C,N) = (w,N) = f$.  That is, we may write $C = f C'$ with $(C',N') = 1$.  The congruence on $D$ follows from $D \equiv \overline{A} \pmod{|C|}$ and $A \equiv - \overline{u} C' \pmod{N'}$, which together give $D \equiv - u \overline{C'} \pmod{(f,N')}$, as claimed.  The condition $C > 0$ may be arranged by multiplication by $-I$.

Finally,
we show that $\overline{\psi}(\sigma_{1/w} \tau)$ only depends on the data appearing in \eqref{eq:CDconditions}.  Explicitly,
\begin{equation}
\label{eq:psiofAformula}
\overline{\psi}(\sigma_{1/w} \tau) = 
\psi(A) = 
\psi^{(N_0')}(-\overline{u} C') \psi^{(f_0)}(\overline{D}).
\end{equation}

We first show that given $C,D \in \mathbb{Z}$ satisfying \eqref{eq:CDconditions}, the value of $\psi(A)$ is uniquely determined.
The determinant condition on $A$ is $A \equiv \overline{D} \pmod{C}$ and the congruence is $A \equiv - \overline{u} C' \pmod{N'}$.  This determines $A$ modulo the least common multiple of $C$ and $N'$, namely $\frac{C N'}{(C,N')} = \frac{CN'}{(f,N')} = CN''$ (one can also see how the left $\Gamma_{\infty}$ action translates $A$ by this).  The condition that these two congruences on $A$ are consistent is precisely the congruence on $D$ in \eqref{eq:CDconditions}.  These two congruences on $A$ uniquely determine $A$ modulo $\frac{C' f N'}{(f,N/f)} = C' \frac{N}{(f,N/f)}$.  Since $\psi$ is periodic modulo $\frac{N}{(f,N/f)}$, this means that $\psi(A)$ is uniquely determined.
Finally, we need to show \eqref{eq:psiofAformula}.

We take an interlude to discuss the problem in more general terms.  Suppose that we have a pair of congruences  $x \equiv a \pmod{Q}$ and $x \equiv b \pmod{R}$, and for consistency, we have $a \equiv b \pmod{(Q,R)}$.  We wish to evaluate $\chi(x)$, where $\chi$ is a Dirichlet character modulo $[Q,R]$.  
There exist integers $Q_0, R_0$ with the following properties:
\begin{equation*}
 Q_0 | Q, \qquad R_0 | R, \qquad [Q,R] = Q_0 R_0, \qquad (Q_0, R_0) = 1.
\end{equation*}
One may check that
\begin{equation*}
 (Q,R) = \frac{Q}{Q_0} \frac{R}{R_0}, 
 \qquad \text{and} \qquad
 (Q_0, Q/Q_0) = 1 = (R_0, R/R_0).
\end{equation*}
The former equation follows from $(Q,R) = \frac{QR}{[Q,R]}$, and the latter follows by noting that a prime power $p^k$ exactly dividing $Q_0$ has either $k=0$ or $p^k$ exactly dividing $Q$ (and similarly for prime powers dividing $R_0$).  We also have that $\frac{R}{R_0} | Q_0$ and $\frac{Q}{Q_0} | R_0$, which is deduced from $(R/R_0, Q_0) = (R/R_0, R_0 Q_0) = (R/R_0, [R,Q]) = R/R_0$, and similarly for the other formula.

Using the above coprimality formulas, the system of congruences is equivalent to $x \equiv a \pmod{Q_0}$ and $x \equiv b \pmod{R_0}$, under the consistency condition $a \equiv b \pmod{(Q,R)}$.  Corresponding to the above notation, we may write $\chi = \chi_{1} \chi_{2}$ where $\chi_{1}$ is modulo $Q_0$ and $\chi_2$ is modulo $R_0$, and then $\chi(x) = \chi_1(a) \chi_2(b)$.  This discussion proves the claim, and completes the proof of the lemma.
We recall for emphasis that the consistency condition is recorded in \eqref{eq:CDconditions}.
\end{proof}

\begin{proof}[Proof of Theorem \ref{thm:EcuspInTermsofEchichi}]
We continue with \eqref{eq:EufInsideLemma}.
The first step is to detect the congruence $D \equiv - \overline{C'} u \pmod{(f,N')}$ with Dirichlet characters; for this, observe that $(DC'u, (f,N')) = 1$ holds from the other listed coprimality conditions.  Thus
\begin{multline}
 E_{\frac{1}{uf}}(z,s,\psi) = \delta_{f=N} y^s + 
\Big[ \frac{\frac{y^s}{(N'')^s}}{\varphi((f,N'))} \sum_{\chi \shortmod{(f,N')}} (\overline{\chi} \overline{ \psi}^{(N_0')})(-u)
\\
 \sum_{\substack{(D, fC') = 1 \\ C' \geq 1 \\ (C',N') = 1 }}
 \frac{(\chi \psi^{(N_0')})(C') (\chi \overline{\psi}^{(f_0)})(D) }{|C'f z + D|^{2s}}
\Big(\frac{|C' fz + D|}{C'fz + D}\Big)^k
\Big] 
 .
\end{multline}

Next we claim that we may omit some of the above coprimality conditions.  The modulus of $\chi \psi^{(N_0')}$ is the least common multiple of $(f,N')$ and $N_0'$, and equals
\begin{equation*}
 \frac{(f, N') N_0'}{(f, N', N_0')} = \frac{(f, N') N_0' f_0}{(f, N_0') f_0} = \frac{N}{(\frac{f}{f_0} f_0,N_0') f_0} 
 = \frac{N}{(f, f_0 N_0')} = \frac{N}{f} = N'.
\end{equation*}
Therefore we may omit the condition $(C',N') = 1$.  A similar calculation shows that the modulus of $\chi \overline{\psi}^{(f_0)}$ is $f$, and that we may omit the condition $(D,f) = 1$.
Thus
\begin{multline*}
 E_{\frac{1}{uf}}(z,s,\psi) = \delta_{f=N} y^s + 
 \Big[
 \frac{\frac{y^s}{(N'')^s}}{\varphi((f,N'))} \sum_{\chi \shortmod{(f,N')}} (\overline{\chi} \overline{\psi}^{(N_0')})(-u)
 \\
 \sum_{\substack{(D, C') = 1 \\ C' \geq 1  }}
 \frac{(\chi \psi^{(N_0')})(C') (\chi \overline{\psi}^{(f_0)})(D) }{|C'f z + D|^{2s}}
\Big(\frac{|C' fz + D|}{C'fz + D}\Big)^k \Big]. 
\end{multline*}
The term $C'=0$ may be returned to the sum, because it only occurs when $N'=1$ (i.e., $f=N$), and then 
consulting \eqref{eq:Echi1chi2def}, we have
\begin{equation}
 E_{\frac{1}{uf}}(z,s,\psi) = \frac{1}{(fN'')^s}
 \frac{1}{\varphi((f,N'))} \sum_{\chi \shortmod{(f,N')}} (\overline{\chi} \overline{\psi}^{(N_0')})(-u)
E_{\chi \psi^{(N_0')}, \chi \overline{\psi}^{(f_0)}}(z,s).
\end{equation}
Applying Lemma \ref{lemma:EisensteinPrimitivevsNonPrimitive}, we have
\begin{multline}
 E_{\frac{1}{uf}}(z,s,\psi) = \frac{1}{(fN'')^s}
 \frac{1}{\varphi((f,N'))} \sum_{\chi \shortmod{(f,N')}} (\overline{\chi} \overline{\psi}^{(N_0')})(-u)
 \frac{L(2s, (\chi \psi^{(N_0')})^* (\chi \overline{\psi}^{(f_0)})^*)}{L(2s, \chi^2 \psi^{(N_0')} \overline{\psi}^{(f_0)} )}
 \\
\sum_{a|f} \sum_{b|N'} \frac{\mu(a) \mu(b) (\chi \psi^{(N_0')})^*(b)  (\chi \overline{\psi}^{(f_0)})^*(a)}{(ab)^s}
E_{(\chi \psi^{(N_0')})^*, (\chi \overline{\psi}^{(f_0)})^*}\Big(\frac{b f }{a q_2} z, s\Big),
\end{multline}
where $q_2$ (say) is the conductor of $(\chi \overline{\psi}^{(f_0)})^*$.

Next we set $(\chi \psi^{(N_0')})^* = \chi_1$ and $(\chi \overline{\psi}^{(f_0)})^* = \chi_2$, where $\chi_i$ is primitive of modulus $q_i$, $i=1,2$.  Note that
\begin{equation}
 \frac{L(2s, (\chi \psi^{(N_0')})^* (\chi \overline{\psi}^{(f_0)})^*)}{L(2s, \chi^2 \psi^{(N_0')} \overline{\psi}^{(f_0)} )}
 = \prod_{p|N} \Big(1- \frac{\chi_1(p) \chi_2(p)}{p^{2s}}\Big)^{-1},
\end{equation}
which only depends on $\chi_1 \chi_2$.  Also, a necessary condition on $\chi_1$ and $\chi_2$ is that $\chi_1 \overline{\chi_2} \sim \psi$.
Therefore, by moving the sum over $\chi$ to the inside, we have
\begin{multline}
\label{eq:EufEchichiWithSumOverchinotsimplifiedYet}
 E_{\frac{1}{uf}}(z,s,\psi) = \frac{1}{(fN'')^s}
 \frac{1}{\varphi((f,N'))} 
 \sum_{q_1 | N'} \sum_{q_2 | f} 
 \sumstar_{\substack{\chi_1 \shortmod{q_1} \\ \chi_2 \shortmod{q_2} \\ \chi_1 \overline{\chi_2} \sim \psi}}
  \overline{\chi_1}(-u)
\prod_{p|N} \Big(1- \frac{\chi_1(p) \chi_2(p)}{p^{2s}}\Big)^{-1}
 \\
\sum_{a|f} \sum_{b|N'} \frac{\mu(a) \mu(b) \chi_1(b)  \chi_2(a)}{(ab)^s}
E_{\chi_1, \chi_2}\Big(\frac{ bf}{aq_2 } z, s\Big)
\sum_{\chi \shortmod{(f,N')}} \delta(\chi, \chi_1, \chi_2, \psi)
\end{multline}
where $\delta(\chi,\chi_1, \chi_2, \psi)$ is the indicator function of
\begin{equation}
\label{eq:ChiPsiConditions}
 (\chi \psi^{(N_0')})^* = \chi_1, \qquad (\chi \overline{\psi}^{(f_0)})^* = \chi_2.
\end{equation}
Our claim is that $\sum_{\chi} \delta(\chi, \chi_1, \chi_2, \psi)=1$ under the conditions appearing in \eqref{eq:EufEchichiWithSumOverchinotsimplifiedYet}, which will 
give \eqref{eq:EcuspInTermsofEchichi}, concluding the proof of the theorem.

Now we prove the claim.  Using that $(f_0, N_0') = 1$ and that a prime divides $N$ iff it divides $f_0 N_0' = [f,N']$, we may uniquely factor $\chi_i = \chi_i^{(f_0)} \chi_i^{(N_0')}$ where $\chi_i^{(f_0)}$ has modulus dividing $f_0$, and $\chi_i^{(N_0')}$ has modulus dividing $N_0'$.  We may also factor $\chi$ in the same way, by $\chi = \chi^{(f_0)} \chi^{(N_0')}$; in addition, we may suppose that $\chi$ is primitive of modulus dividing $ (f,N')$.
Recall also that $\psi^{(N_0')}$ has modulus $N_0'$, and $\psi^{(f_0)}$ has modulus $f_0$.  Thus the assumption $\chi_1 \overline{\chi_2} \sim \psi$ is equivalent to
\begin{equation}
\label{eq:chi1chi2psi1psi2conditions}
 \chi_1^{(N_0')} \overline{\chi_2}^{(N_0')} \sim \psi^{(N_0')}, 
 \qquad \text{and} \qquad
 \chi_1^{(f_0)} \overline{\chi_2}^{(f_0)} \sim \psi^{(f_0)}.
\end{equation}
The condition $(\chi \psi^{(N_0')})^* = \chi_1$ from \eqref{eq:ChiPsiConditions} is in turn equivalent to
\begin{equation}
 \chi^{(f_0)} (\chi^{(N_0')} \psi^{(N_0')})^* = \chi_1^{(f_0)} \chi_1^{(N_0')},
\end{equation}
that is,
\begin{equation}
\chi^{(f_0)} = \chi_1^{(f_0)} \qquad \text{and} \qquad (\chi^{(N_0')} \psi^{(N_0')})^* = \chi_1^{(N_0')}.
\end{equation}
Likewise, for the equation with $\chi_2$, we obtain
\begin{equation}
\chi^{(N_0')} = \chi_2^{(N_0')} \qquad \text{and} \qquad (\chi^{(f_0)} \overline{\psi}^{(f_0)})^* = \chi_2^{(f_0)}.
\end{equation}
From these two displayed equations, we see that $\chi$ is uniquely determined by $\chi = \chi_1^{(f_0)} \chi_2^{(N_0')}$.  Once this choice is made, one can check that \eqref{eq:ChiPsiConditions} holds using \eqref{eq:chi1chi2psi1psi2conditions}.  The only remaining loose end is to check that this purported choice of $\chi = \chi_1^{(f_0)} \chi_2^{(N_0')}$ has modulus dividing $(f,N')$.
 That is, we need that the $f_0$-part of $q_1$ divides $(f,N')$, and similarly that the $N_0'$-part of $q_2$ divides $(f,N')$.  
Note that
\begin{equation*}
 N' = \underbrace{N_0'}_{\text{$N_0'$-part}} \times \underbrace{\frac{N'}{N_0'}}_{\text{$f_0$-part}},
 \qquad
 f = \underbrace{f_0}_{\text{$f_0$-part}} \times \underbrace{\frac{f}{f_0}}_{\text{$N_0'$-part}},
\end{equation*}
and also that
\begin{equation*}
 (f,N') = 
 \underbrace{\frac{f}{f_0}}_{\text{$N_0'$-part}}
 \times
 \underbrace{\frac{N'}{N_0'}}_{\text{$f_0$-part}}.
\end{equation*}
These equations show that the $f_0$-part of $N'$ equals the $f_0$-part of $(f,N')$ (both are equal to $\frac{N'}{N_0'}$), and so we conclude that the $f_0$-part of $q_1$ divides $(f,N')$.  A similar argument holds for the $N_0'$-part for the other factor.  This shows the claim, and completes the proof of Theorem \ref{thm:EcuspInTermsofEchichi}.
\end{proof}

\section{Inversion} 
 The purpose of this section is to invert \eqref{eq:EcuspInTermsofEchichi}, which is given by the following:
 \begin{mytheo}
 \label{thm:EchichiInTermsofEa}
  Let $\chi_i$, $i=1,2$, be primitive characters modulo $q_i$ with $q_1 q_2 |N$, and write $N = q_1 q_2 L$.  Suppose $B |L$, and write $L = AB$.  Then 
  \begin{equation}
  \label{eq:EchichiInTermsofEa}
E_{\chi_1, \chi_2}(Bz, s) =  
 \mathop{\sum_{d|A} \sum_{e|B}}_{(d,e)=1} \frac{\chi_1(d) \chi_2(e)}{(de)^s}
 \\
 \Big(\frac{N}{(q_2 \frac{Bd}{e}, q_1 \frac{Ae}{d})}\Big)^s 
 \sumstar_{u \shortmod{ (q_2 \frac{Bd}{e}, q_1 \frac{Ae}{d})}} \chi_1(-u)  E_{\frac{1}{u q_2 \frac{Bd}{e}}}(z,s, \psi).
\end{equation}
Here the sum is over $u$ is over a set of representatives for $(\mathbb{Z}/(q_2 \frac{Bd}{e}, q_1 \frac{Ae}{d}) \mathbb{Z})^*$, chosen coprime to $N$, and $\psi$ is modulo $N$, induced by $\chi_1 \overline{\chi_2}$.
\end{mytheo}

Remark.   Note that $\frac{Bd}{e}$ ranges over certain divisors of $L$, so that $E_{\chi_1,\chi_2}(Bz,s)$ is a linear combination of $E_{\frac{1}{uf}}$'s with $f$'s constrained by $q_2 | f$ and $f | q_2 L$.

Within the proof of Theorem \ref{thm:EchichiInTermsofEa}, we shall develop and use properties of functions $D_{\chi_1, \chi_2,f}(z,s, \psi)$ defined by
\begin{equation}
\label{eq:DchiDef}
 D_{\chi_1, \chi_2,f}(z,s, \psi) = \sumstar_{u \shortmod{(f, N/f)}} \chi_1(-u)  E_{\frac{1}{uf}}(z,s, \psi),
\end{equation}
where $\chi_i$ is primitive modulo $q_i$, $N = q_1 q_2 L$, $\chi_1 \overline{\chi_2} \sim \psi$,  $q_2 | f$ and $f | q_2 L$.   Notice that \eqref{eq:EchichiInTermsofEa} may be expressed as 
\begin{equation}
\label{eq:EchichiInTermsofDchi}
 E_{\chi_1, \chi_2}(Bz, s) =  N^s \mathop{\sum_{d|A} \sum_{e|B}}_{(d,e)=1} \frac{\chi_1(d) \chi_2(e)}{(de)^s} 
\frac{1}{(q_2 \frac{Bd}{e}, q_1 \frac{Ae}{d})^s} 
 D_{\chi_1, \chi_2, q_2 \frac{Bd}{e}}(z,s,\psi).
\end{equation}
It is not obvious from \eqref{eq:DchiDef} that $D_{\chi_1,\chi_2,f}$ is well-defined.  To see this, first note 
\begin{equation}
\label{eq:chi1IsWellDefined}
\chi_1 = \chi_1^{(f_0)} \chi_1^{(N_0')} = \psi^{(N_0')} \chi_1^{(f_0)} \chi_2^{(N_0')}.
\end{equation}
Now, $\psi^{(N_0')}(u) E_{\frac{1}{uf}}$ is well-defined, as observed in the paragraph preceding Lemma \ref{lemma:EcuspIntermediateFormula}.  In addition, one may directly check that $\chi_1^{(f_0)}$ is periodic modulo $f$ (since $f_0|f$) and modulo $N'$ (since $q_1 | N'$), and is therefore periodic modulo $(f,N')$.  A similar argument holds for $\chi_2^{(N_0')}$.

\begin{proof}
Let $q_1, q_2, L, A, B$ be as in the statement of the theorem.  Set
$f = q_2 g$, where $g|L$. Then 
 $N' = N/f = q_1 L/g$,
$(f,N') =  (q_2 g, q_1 \frac{L}{g})$, and $N'' = \frac{N'}{(N',f)}  = \frac{q_1 \frac{L}{g}}{(q_2 g, q_1 \frac{L}{g} )}$.  Our first step is to derive a formula for $D_{\chi_1, \chi_2, f}(z,s,\psi)$ by inserting \eqref{eq:EcuspInTermsofEchichi} into the definition \eqref{eq:DchiDef}, giving
\begin{multline*}
D_{\chi_1, \chi_2, f}(z,s,\psi)= 
 \frac{1}{(fN'')^s}
 \frac{1}{\varphi((f,N/f))} 
 \sum_{k_1 | \frac{N}{f}} \sum_{k_2 | f} 
 \thinspace
 \sumstar_{\substack{\eta_1 \shortmod{k_1} \\ \eta_2 \shortmod{k_2} \\ \eta_1 \overline{\eta_2} \sim \psi}}
  \frac{L(2s,\eta_1 \eta_2)}{L(2s,\eta_1 \eta_2 \chi_{0,N})}
 \\
\sum_{\substack{a|f \\ (a, k_2)=1}} \sum_{\substack{b|\frac{N}{f} \\ (b, k_1) = 1}} \frac{\mu(a) \mu(b) \eta_1(b)  \eta_2(a)}{(ab)^s}
E_{\eta_1, \eta_2}\Big(\frac{b f}{a k_2} z, s\Big)
  \sumstar_{u \shortmod{(f, N/f)}}  
  \chi_1(-u) \overline{\eta_1}(-u).
\end{multline*}
%
%

We claim the inner sum over $u$ equals $\varphi((f,N'))$ if $\chi_1 = \eta_1$, and vanishes otherwise.  For this, apply \eqref{eq:chi1IsWellDefined} to both $\chi_1$ and $\eta_1$, which implies that the sum vanishes unless $\chi_1^{(f_0)} = \eta_1^{(f_0)}$ and $\chi_2^{(N_0')} = \eta_2^{(N_0')}$, recalling that the characters are primitive and that $(f_0,N_0') = 1$.  Using \eqref{eq:chi1IsWellDefined} again, we deduce that $\chi_1 = \eta_1$.
It is also necessary to verify that the value $k_1 = q_1$ does indeed occur in the sum, which follows from $\frac{N}{f} = q_1 \frac{L}{g}$ and $g|L$; similarly, $k_2 = q_2$ occurs since $q_2 | f$.

We may next see that from $\eta_1 \overline{\eta_2} \sim \psi \sim \chi_1 \overline{\chi_2}$ that $\eta_2 = \chi_2$ (whence $k_2 = q_2$).
Thus
\begin{equation}
\label{eq:Etwistedavgoveru}
D_{\chi_1, \chi_2, q_2 g}(z,s,\psi)
 =
 \frac{1}{(f  N'')^s} 
 \frac{L(2s,\chi_1 \chi_2)}{L(2s,\chi_1 \chi_2 \chi_{0,N})}
 \sum_{\substack{a|g}}  \sum_{\substack{b | \frac{L}{g} }}
 \frac{\mu(a) \mu(b)}{(ab)^{s}}
  \chi_1(b) \chi_2(a)  
 E_{\chi_1,\chi_2}\Big(\frac{bg}{a } z, s\Big),
\end{equation}
using additionally that $a|f$ may be replaced by $a|g$ and similarly $b|\frac{L}{g}$ since $(a,q_2) = 1$ and $(b,q_1) = 1$.

Next we interject an elementary inversion formula for certain arithmetical functions.
\begin{mylemma}
\label{lemma:InversionElementaryLemma}
 Let $\omega_i$, $i=1,2$ be completely multiplicative functions 
 and suppose $K$ is an arbitrary function defined on the divisors of some positive integer $L$.  For
$g | L$, define
 \begin{equation}
 \label{eq:JintermsofK}
  J(g) = \sum_{a|g} \sum_{b|\frac{L}{g}} \mu(a) \mu(b) \omega_2(a) \omega_1(b)  K\Big(\frac{b g}{a}\Big).
 \end{equation}
 Then with $AB=L$, we have
\begin{equation}
\label{eq:KintermsofJ}
 K(B) = \Big(\prod_{p|L} (1- \omega_1 \omega_2(p))^{-1} \Big) \mathop{\sum_{d|A} \sum_{e|B}}_{(d,e) = 1} \omega_{1}(d) \omega_{2}(e) J\Big(\frac{Bd}{e}\Big).
\end{equation}
\end{mylemma}
We defer the proof of the lemma to Section \ref{section:Inversion} in order to complete the proof of Theorem \ref{thm:EchichiInTermsofEa}.  We apply the lemma with $J(g) = ( \frac{N}{(q_2 g, q_1 \frac{L}{g})})^s \frac{L(2s,\chi_1 \chi_2 \chi_{0,N})}{L(2s,\chi_1 \chi_2)} D_{\chi_1, \chi_2, q_2 g}(z,s,\psi)$, $K(B) = E_{\chi_1,\chi_2}(Bz,s)$, and $\omega_i(n) = \chi_i(n) n^{-s}$, obtaining
\begin{multline*}
 E_{\chi_1,\chi_2}(Bz,s) = \prod_{p|L} (1-p^{-2s} \chi_1(p) \chi_2(p))^{-1}
 \mathop{\sum_{d|A} \sum_{e|B}}_{(d,e)=1} \frac{\chi_1(d) \chi_2(e)}{(de)^s}
 \\
 \Big(\frac{N}{(q_2 \frac{Bd}{e}, q_1 \frac{Ae}{d})}\Big)^s \frac{L(2s, \chi_1 \chi_2 \chi_{0,N})}{L(2s, \chi_1 \chi_2)}
 D_{\chi_1, \chi_2,q_2 \frac{Bd}{e}}(z,s).
\end{multline*}
Note that $\prod_{p|L} (1-p^{-2s} \chi_1(p) \chi_2(p))^{-1} = \prod_{p|N} (1-p^{-2s} \chi_1(p) \chi_2(p))^{-1}$, so this factor cancels the ratio of Dirichlet $L$-functions.  Inserting \eqref{eq:DchiDef} into the above formula for $E_{\chi_1,\chi_2}(Bz,s)$ and simplifying, we obtain the theorem. 
\end{proof}

\section{Orthogonality properties}
\label{section:orthogonality}
\subsection{Orthogonal decomposition into newforms}
\label{section:orthogonaldecompositionintonewforms}
With Theorems \ref{thm:EcuspInTermsofEchichi} and  \ref{thm:EchichiInTermsofEa} in hand, we may now 
study the orthogonality properties of Eisenstein series attached to Dirichlet characters. 
Let $\mathcal{E}_{t,\psi}(N)$ be the finite-dimensional vector space defined by
\begin{equation*}
 \mathcal{E}_{t,\psi}(N) = \text{span} \{ E_{\a}(z, 1/2 + it, \psi): \a \text{ is singular for $\psi$}\},
\end{equation*}
and define a formal inner product $\langle , \rangle_{\text{Eis}}$ on this space
by
\begin{equation}
\label{eq:FormalInnerProductOrthogonalityEquation}
\frac{1}{4 \pi} \langle E_{\a}(\cdot, 1/2 + it, \psi), E_{\b}(\cdot, 1/2 + it, \psi) \rangle_{\text{Eis}} = \delta_{\a \b},
\end{equation}
extended bilinearly.  This inner product is natural to use since the spectral decomposition of $L^2(\Gamma_0(N), \psi)$ in terms of Eisenstein series attached to cusps (as in \cite[Propositions 4.1, 4.2]{DFI}) corresponds essentially to \eqref{eq:FormalInnerProductOrthogonalityEquation}; see Section \ref{section:SpectralDecomposition} for more discussion.

Perhaps it is worthy of explanation that the dimension of $\mathcal{E}_{t,\psi}(N)$ equals the number of singular cusps for $\psi$, except possibly for $t=0$, and therefore this inner product is well-defined.
The key is to study the Fourier expansion of $E_{\a}(z,1/2+it, \psi)$ at the various cusps.  One may easily show that if $c y^{1/2+it} + d y^{1/2-it} = c' y^{1/2+it} + d' y^{1/2-it}$ for infinitely many values of $y$, and $t \neq 0$, then $c=c'$ and $d=d'$. 
Now suppose that
\begin{equation*}
 E_{\a}(z,1/2+it,\psi) = \sum_{\b} c_{\a,\b} E_{\b}(z,1/2+it, \psi),
\end{equation*}
for some constants $c_{\a, \b}$.  Equating coefficients of $y^{1/2+it}$ in the Fourier expansions at the arbitrary cusp $\c$, we have
\begin{equation*}
\delta_{\a=\c} y^{1/2 + it} 
= \sum_{\b} c_{\a, \b}  \delta_{\b = \c} y^{1/2+it} 
= c_{\a, \c} y^{1/2+it}.
\end{equation*}
Hence
 $c_{\a, \c} = \delta_{\a = \c}$,
which precisely means that the Eisenstein series attached to cusps are linearly independent, for $t \neq 0$.  We should also observe that the constant terms of all Eisenstein series are analytic for $s=1/2+it$, except possibly at $t=0$, by inspection of \eqref{eq:EisensteinChiChiConstantTerm}.

\begin{mytheo}
\label{thm:DorthogonalBasis}
For $f|N$, $(f,N/f)|\frac{N}{N^*}$, $q_1 | \frac{N}{f}$, $q_2 | f$, and $\chi_i$ primitive modulo $q_i$ satisfying $\chi_1 \overline{\chi_2} \sim \psi$, let
 $D_{\chi_1, \chi_2, f}(z,s)$ be defined by \eqref{eq:DchiDef}.  Then the functions $D_{\chi_1, \chi_2, f}(z,s)$ form an orthogonal basis for $\mathcal{E}_{t,\psi}(N)$.
\end{mytheo}

\begin{proof}
These functions are defined by \eqref{eq:DchiDef}, but also may be given by \eqref{eq:Etwistedavgoveru}.  The formula \eqref{eq:EchichiInTermsofDchi} allows one to express $E_{\chi_1, \chi_2}$ in terms of $D$'s, while \eqref{eq:DchiDef} may be inverted by inserting \eqref{eq:Etwistedavgoveru} into \eqref{eq:EcuspInTermsofEchichi}, giving
\begin{equation}
\label{eq:EaintermsofDchiAlternate}
 E_{\frac{1}{uf}}(z,s,\psi) =
  \frac{1}{\varphi((f,N/f))} 
 \sum_{q_1 | \frac{N}{f}} \sum_{q_2 | f} 
 \sumstar_{\substack{\chi_1 \shortmod{q_1} \\ \chi_2 \shortmod{q_2} \\ \chi_1 \overline{\chi_2} \sim \psi}}
  \overline{\chi_1}(-u) D_{\chi_1, \chi_2, f}(z,s,\psi).
\end{equation}
This formula shows the functions $D_{\chi_1, \chi_2, f}$ form a spanning set for $\mathcal{E}_{t,\psi}(N)$.

To show these functions are orthogonal, we simply combine \eqref{eq:DchiDef} and \eqref{eq:FormalInnerProductOrthogonalityEquation}, giving
\begin{equation*}
\frac{1}{4 \pi} \langle D_{\chi_1, \chi_2, f_1}, D_{\eta_1, \eta_2, f_2} \rangle = \delta_{f_1 = f_2} \delta_{\chi_1 = \eta_1} \delta_{\chi_2 = \eta_2} \varphi((f, N/f)),
\end{equation*}
where we have let $f = f_1 = f_2$.  
In particular, for a fixed $f$ for which $(f,N/f)|\frac{N}{N^*}$  (equivalently, $N^*|[f,N/f]$), there is a bijection between $u \pmod{(f,N/f)}$, coprime to the modulus, and pairs of characters $\chi_1, \chi_2$ so that $\chi_1 \overline{\chi_2} \sim \psi$.  That is, the number of such pairs of characters equals $\varphi((f,N/f))$.  
\end{proof}

Next we turn to the orthogonality properties of $E_{\chi_1, \chi_2}$.
We deduce from \eqref{eq:EchichiInTermsofDchi} that $E_{\chi_1, \chi_2}(B_1 z, 1/2 + it)$ is orthogonal to $E_{\eta_1, \eta_2}(B_2 z, 1/2 + it)$ unless $\chi_1 = \eta_1$ and $\chi_2 = \eta_2$.  This shows that, for a given $\chi_1, \chi_2$, the set of functions $D_{\chi_1, \chi_2, q_2 g}(z,s)$ with $g|L$ forms an orthogonal set for the ``oldclass" formed from $E_{\chi_1, \chi_2}(z,s)$, that is, the subspace
\begin{equation*}
\mathcal{E}_{t,\psi}(L;E_{\chi_1, \chi_2}) := \text{span} \{ E_{\chi_1, \chi_2}(Bz,1/2+it): B | L \},
\end{equation*}
where $q_1 q_2 L = N$.
By dimension-counting, we see that $\{D_{\chi_1, \chi_2, q_2 g}(z,s):g|L \}$ then forms an orthogonal basis for this oldclass, and so the functions $E_{\chi_1, \chi_2}(Bz,1/2+it)$ also form a basis for this subspace (not in general orthogonal, however).

 Summarizing, we have shown
\begin{equation}
\label{eq:EisensteinAtkinLehnerDecomposition}
 \mathcal{E}_{t,\psi}(N)
 = \bigoplus_{q_1 q_2 L =N} \thinspace \sideset{}{^*}\bigoplus_{\substack{\chi_1 \shortmod{q_1} \\ \chi_2 \shortmod{q_2} \\ \chi_1 \overline{\chi_2} \sim \psi}}  
 \mathcal{E}_{t,\psi}(L;E_{\chi_1,\chi_2}),
\end{equation}
where, as observed earlier in this section, this is an orthogonal decomposition. 
This is an extension of Weisinger's newform theory to the non-holomorphic setting.
 Following Weisinger, define an Eisenstein newform of level $M$ to be one of the  $E_{\chi_1, \chi_2}(z, 1/2+it)$, where $\chi_i$ is primitive modulo $q_i$, $i=1,2$, with $q_1 q_2 = M$.  Let $\mathcal{H}_{t,\psi}^*(M)$ denote the set of Eisenstein newforms of level $M$, nebentypus $\psi$, and spectral parameter $t$.  Then we may re-write \eqref{eq:EisensteinAtkinLehnerDecomposition} as
\begin{equation}
\label{eq:EisensteinAtkinLehnerDecomposition2}
 \mathcal{E}_{t,\psi}(N)
 = \bigoplus_{LM = N} \bigoplus_{F \in \mathcal{H}_{t,\psi}^*(M)} \mathcal{E}_{t,\psi}(L;F).
\end{equation}


Needless to say, the above decompositions completely parallel the decomposition of cuspidal newforms as in \cite{AtkinLehner} \cite{AtkinLi}, which gives
\begin{equation*}
 S_{t_j, \psi}(N) = \bigoplus_{LM = N} \bigoplus_{f \in H_{t_j,\psi}^*(M)} S_{t_j, \psi}(L;f),
\end{equation*}
where $S_{t_j,\psi}(N)$ is the (finite-dimensional) space of cusp forms with spectral parameter $t_j$ and nebentypus $\psi$,  $H_{t_j,\psi}^*(M)$ is the set of newforms of level $M$ with spectral parameter $t_j$, and $S_{t_j,\psi}(L;f)= \text{span}\{f(\ell z): \ell | L \}$.

\subsection{Summary remarks}
For clarity, we summarize the statements of the change-of-basis formulas with some alternative notation.  We take this opportunity to make explicit certain facts that were perhaps only implicit within the proofs.

Let $\psi$ be a Dirichlet character modulo $N$, of conductor $N_{\psi}$.  The cusp $\frac{1}{uf} \sim \frac{u}{f}$ is singular with respect to $\psi$ iff 
$ (f,N/f) | \frac{N}{N_{\psi}}$,
or alternatively, $N_{\psi} | [f,N/f]$.
There are $\varphi((f,N/f))$ inequivalent (singular) cusps $u/f$ with denominator $f$.  Moreover, there exists a bijection between these cusps, and pairs of characters $(\chi_1, \chi_2)$ with $q_1 | \frac{N}{f}$, $q_2 | f$, $\chi_i$ primitive modulo $q_i$, $i=1,2$, and $\chi_1 \overline{\chi_2} \sim \psi$.  
This bijection was observed within the proof of Theorem \ref{thm:DorthogonalBasis} by comparing dimensions, but can be seen directly as follows.  The parts of $\chi_1, \chi_2$ of moduli away from $(f,N/f)$ are uniquely determined by the equation $\chi_1 \overline{\chi_2} \sim \psi$ (and there will exist at least one such pair of characters, since $N_{\psi} | [f,N/f]$).  After that, we are free to multiply both $\chi_1$ and $\chi_2$ by the same Dirichlet character modulo $(f,N/f)$. 
Let $\Psi_f$ denote the set of pairs of such characters, so $|\Psi_f| = \varphi((f,N/f))$.   

For $(\chi_1, \chi_2) \in \Psi_f$, we may define $D_{\chi_1,\chi_2,f}(z,s,\psi)$ by \eqref{eq:DchiDef}. 
This formula is inverted by \eqref{eq:EaintermsofDchiAlternate}, which in the new notation reads
\begin{equation}
\label{eq:EaInTermsofDchi}
 E_{\frac{1}{uf}}(z,s,\psi)
 = \frac{1}{\varphi((f,N/f))} \sum_{(\chi_1, \chi_2) \in \Psi_f} \overline{\chi_1}(-u) D_{\chi_1, \chi_2, f}(z,s,\psi).
\end{equation}

One may wish to focus on the pair of characters themselves intrinsically, and to forget about the ambient $f$.  
Suppose that $q_1 q_2 | N$, say $N = q_1 q_2 L$,  $\chi_i$ is primitive of modulus $q_i$, $i=1,2$, and $\chi_1 \overline{\chi_2} \sim \psi$.  We claim that $(\chi_1, \chi_2) \in \Psi_f$ if and only if $f = q_2 g$ with $g | L$.  This is easy to check, because the condition $q_2 | f$ means that $f = q_2 g$ for some $g$, and since $\frac{N}{f}  = \frac{q_1 L}{ g} $,
the condition $q_1 | \frac{N}{f}$ means $g | L$.  In particular, there always exists such an $f$ so that $(\chi_1, \chi_2) \in \Psi_f$.  Moreover, the same character pair $(\chi_1, \chi_2)$ lies in $\tau(L)$ sets $\Psi_{f}$.

Now suppose that $(\chi_1, \chi_2) \in \Phi_{q_2 g}$ with $g |L$.  
Then \eqref{eq:Etwistedavgoveru} becomes
\begin{equation*}
 D_{\chi_1, \chi_2, q_2 g}(z,s,\psi)
 =
 \frac{(q_2 g, q_1 \frac{ L}{g})^s}{N^s} 
 \frac{L(2s,\chi_1 \chi_2)}{L(2s,\chi_1 \chi_2 \chi_{0,N})}
 \sum_{\substack{a|g}}  \sum_{\substack{b | \frac{L}{g} }}
 \frac{\mu(a) \mu(b)}{(ab)^{s}}
  \chi_1(b) \chi_2(a)  
 E_{\chi_1,\chi_2}\Big(\frac{bg}{a } z, s\Big),
\end{equation*}
which is inverted by \eqref{eq:EchichiInTermsofDchi}.

The $D$-functions are useful because they may be naturally parameterized either by $f$ and $\Psi_f$ or alternatively by the (intrinsic) pairs of characters, along with $g|L$.  Hence, they give a natural intermediate basis between the $E_{\frac{1}{uf}}$ and the $E_{\chi_1, \chi_2}$.

\subsection{Remarks on the spectral decomposition}
\label{section:SpectralDecomposition}
The continuous part of the spectral decomposition, as in \cite[Proposition 4.1]{DFI}, for instance, takes the form
\begin{equation}
\label{eq:SpectralDecompositionContinuousPart}
f_{\text{Eis}}(z) :=  \intR \sum_{\mathfrak{a}} \frac{1}{4 \pi} \langle f, E_{\mathfrak{a}} \rangle E_{\mathfrak{a}}(z, 1/2 + it, \psi) dt,
\end{equation}
where $f \in L^2(\Gamma_0(N), \psi)$.  It is desirable to express this formula in terms of an alternative basis, as in, for instance, Section \ref{section:orthogonaldecompositionintonewforms}, without having to go through the analytic aspects of the spectral decomposition.
\begin{myprop}
\label{prop:SpectralFormulasChangeOfBasis}
 Let $\mathcal{B}(t,\psi,N)$ denote an orthogonal basis for $\mathcal{E}_{t,\psi}(N)$, and suppose $f, g \in L^2(\Gamma_0(N), \psi)$.  Then
 \begin{equation}
 \label{eq:SpectralFormulasGeneralBasis}
  \sum_{F \in \mathcal{B}(t,\psi,N)} \frac{\langle f, F \rangle}{\langle F, F \rangle_{\text{Eis}}} F(z),
  \quad
  \text{and}
  \quad
  \sum_{F \in \mathcal{B}(t,\psi,N)} \frac{\langle f, F \rangle \langle F, g \rangle}{\langle F, F \rangle_{\text{Eis}}} 
 \end{equation}
are independent of the choice of basis.
 \end{myprop}
Note that with $F = E_{\a}$, we have $\langle F, F \rangle_{\text{Eis}} = 4 \pi$, and the first expression in \eqref{eq:SpectralFormulasGeneralBasis} agrees with the integrand in \eqref{eq:SpectralDecompositionContinuousPart}.  Likewise, the second formula in \eqref{eq:SpectralFormulasGeneralBasis} is the continous spectrum part of $\langle f, g \rangle$ in Parseval's formula.
These formulas are not quite the standard formulas for the projection of a vector $f$ onto a finite-dimensional inner product space, and $\langle f, g \rangle$, respectively, because the inner products in the numerators are different from the inner products in the denominators.  Nevertheless, the formulas follow from standard linear algebra calculations.
\begin{proof}
 Let $G$ run over an alternative basis, say $\mathcal{B}'(t,\psi,N)$, and define the change of basis coefficients by $F = \sum_{G} c_{F,G} G$, 
 where
 $c_{F,G} = \frac{\langle F, G \rangle_{\text{Eis}}}{\langle G, G \rangle_{\text{Eis}}}$.  Note that, if $G,G' \in\mathcal{B}'(t,\psi,N)$, then
 \begin{equation}
 \label{eq:cFGsumFormula}
  \sum_{F} \frac{\overline{c_{F,G}} c_{F, G'}}{\langle F, F \rangle_{\text{Eis}}} 
  = 
  \sum_{F} \frac{\langle G, F \rangle_{\text{Eis}} \langle F, G' \rangle_{\text{Eis}} }{\langle G, G \rangle_{\text{Eis}} \langle G', G' \rangle_{\text{Eis}} \langle F, F \rangle_{\text{Eis}}} 
  = \frac{\langle G, G' \rangle_{\text{Eis}}}{\langle G, G \rangle_{\text{Eis}}\langle G', G' \rangle_{\text{Eis}}}
  =  \frac{\delta_{G=G'}}{\langle G, G \rangle_{\text{Eis}}}.
 \end{equation}
Applying \eqref{eq:cFGsumFormula}, we have
\begin{equation}
 \sum_{F} \frac{\langle f, F \rangle}{\langle F, F \rangle_{\text{Eis}}} F 
 =
 \sum_{F} \sum_{G}  \frac{\langle f, G \rangle}{\langle F, F \rangle_{\text{Eis}}} \overline{c_{F,G}} \sum_{G'} c_{F,G'} G' =
 \sum_{G} \frac{\langle f, G \rangle}{\langle G, G \rangle_{\text{Eis}}} G,
\end{equation}
showing that the first formula in \eqref{eq:SpectralFormulasGeneralBasis} is independent of basis.
A nearly-identical proof works for the second formula in \eqref{eq:SpectralFormulasGeneralBasis}.
\end{proof}

\subsection{An inner product calculation}
\label{section:InnerProduct}
Let $M= q_1 q_2$, $N = ML$, and let $\chi_i$ be primitive modulo $q_i$.  We wish to evaluate 
\begin{equation}
\label{eq:Ichi1chi2B1B2Definition}
I_{\chi_1,\chi_2}(B_1,B_2;N):= \frac{1}{4\pi} \langle E_{\chi_1, \chi_2}(B_1 z,1/2+it), E_{\chi_1, \chi_2}(B_2 z,1/2+it) \rangle_N, 
\end{equation}
where $B_1, B_2 | L$, and the inner product is on Eisenstein series of level $N$.

The motivation to evaluate this inner product is to unify it with a corresponding formula for cuspidal newforms, for which see \cite[p.473]{BlomerMilicevic} (for principal nebentypus)
and \cite[Lemma 3.13]{Humphries} (for arbitrary nebentypus).  Schulze-Pillot and Yenirce \cite{SchulzePillotYenirce} have also derived the analogous formula for holomorphic newforms of arbitrary level and nebentypus, using only Hecke theory.
This is desirable in order to find orthonormal bases for the oldclasses $S_{t_j, \psi}(L;f)$ and $\mathcal{E}_{t,\psi}(L;F)$
that are constructed from the newforms in identical ways, which is useful to treat the discrete spectrum and the continuous spectrum on an equal footing.
In Section \ref{section:KuznetsovNewforms} below, we illustrate this idea by proving a Bruggeman-Kuznetsov formula for newforms of squarefree level and trivial nebentypus (these restrictions on the level and nebentypus arise from the assumptions in place in \cite{PetrowYoung}).

\begin{mylemma}
\label{lemma:InnerProductEvaluation}
 Let notation be as above.  Then
 \begin{equation}
 \label{eq:Ichi1chi2InnerProductinTermsofA}
  \frac{I_{\chi_1,\chi_2}(B_1,B_2;N)}{I_{\chi_1,\chi_2}(1,1;N)}
  = A_{\chi_1, \chi_2}\Big(\frac{B_2}{(B_1,B_2)} \Big)
  \overline{A_{\chi_1, \chi_2}}\Big(\frac{B_1}{(B_1,B_2)} \Big),
 \end{equation}
 where $A_{\chi_1,\chi_2}(n)$ is the multiplicative function defined for $B \geq 1$ by
 \begin{equation}
 \label{eq:Achi1chi2PrimePowerFormula}
  A_{\chi_1,\chi_2}(p^B) = \frac{\lambda_{\chi_1,\chi_2}(p^B) - \chi_1 \overline{\chi_2}(p)  p^{-1} \lambda_{\chi_1,\chi_2}(p^{B-2})}{p^{B/2} (1 + \chi_0(p) p^{-1})}.
 \end{equation}
 Here $\lambda_{\chi_1,\chi_2}(n)$ is shorthand for $\lambda_{\chi_1, \chi_2}(n, 1/2+it)$ originally defined by \eqref{eq:lambdachi1chi2Def}, and where for $B=1$ we define $\lambda_{\chi_1, \chi_2}(p^{-1}) = 0$.  Moreover, $\chi_0$ is the principal character modulo $q_1 q_2$.
\end{mylemma}
The form of \eqref{eq:Achi1chi2PrimePowerFormula} is in perfect accord with the cuspidal case of \cite[Lemma 3.13]{Humphries}.  

The method of Blomer and Mili\'cevi\'c proceeds by unfolding and Rankin-Selberg theory; this method may not be used for Eisenstein series due to the lack of convergence.  As a substitute, we use the change-of-basis formulas and orthogonality of $E_{\a}$'s.  

\begin{proof}
 
Write $A_i = L/B_i$, $i=1,2$.  Then from Theorem \ref{thm:EchichiInTermsofEa}, we have
\begin{equation*}
I_{\chi_1,\chi_2}(B_1,B_2;N) 
= 
N \mathop{\mathop{\sum_{d_1|A_1} \sum_{e_1|B_1}}_{(d_1,e_1)=1}
\mathop{\sum_{d_2|A_2} \sum_{e_2|B_2}}_{(d_2,e_2)=1} 
}_{\frac{B_1 d_1}{e_1} = \frac{B_2 d_2}{e_2}}
\frac{\chi_1(d_1 \overline{d_2}) \chi_2(e_1 \overline{e_2})}{(d_1 e_1 d_2 e_2)^{1/2}} \Big(\frac{d_2 e_2}{d_1 e_1} \Big)^{it} 
\frac{\varphi((q_2 \frac{B_1 d_1}{e_1}, q_1 \frac{A_1 e_1}{d_1}))}{(q_2 \frac{B_1 d_1}{e_1}, q_1 \frac{A_1 e_1}{d_1})^{}}.
\end{equation*}
Parameterizing by the value of $\frac{B_1 d_1}{e_1}$, we have
\begin{equation}
\label{eq:Ichi1chi2B1B2multiplicativeExpression}
 I_{\chi_1,\chi_2}(B_1,B_2;N) 
= 
N
\sum_{R|L} 
\frac{\varphi((q_2 R, q_1 \frac{L}{R}))}{(q_2 R, q_1 \frac{L}{R})}
G(B_1,R) \overline{G_2(B_2,R)},
\end{equation}
where
\begin{equation*}
 G(B_i,R) = \mathop{\sum_{d | A_i} \sum_{e | B_i}}_{\substack{(d,e) = 1 \\ \frac{B_i d}{e} = R}} 
 \frac{\chi_1(d) \chi_2(e)}{(de)^{1/2+it}}.
\end{equation*}
Based on the multiplicative structure of \eqref{eq:Ichi1chi2B1B2multiplicativeExpression}, we may write
\begin{equation*}
 I_{\chi_1,\chi_2}(B_1, B_2;N) = N \prod_{p|L} I^{(p)}_{\chi_1,\chi_2}(p^{\nu_p(B_1)}, p^{\nu_p(B_2)}),
\end{equation*}
say.  
By abuse of notation, we replace $\nu_p(B_i)$ by $B_i$, and focus on a single prime $p$.
We have
\begin{equation}
\label{eq:GipR}
 G(p^{B_i},p^R) =  \mathop{\sum_{0 \leq d \leq A_i} \sum_{0 \leq e \leq B_i}}_{\substack{(p^d,p^e) = 1 \\ B_i+d-e = R}} 
 \frac{\chi_1(p^d) \chi_2(p^e)}{p^{(d+e)(1/2+it)}},
\end{equation}
and 
\begin{equation}
\label{eq:IpDefinition}
I_{\chi_1,\chi_2}^{(p)}(p^{B_1},p^{B_2})  = \sum_{0 \leq R \leq L} 
\frac{\varphi((p^{q_2+ R}, p^{q_1 +L-R}))}{(p^{q_2+ R}, p^{q_1 +L-R})}
G(p^{B_1},p^{R}) \overline{G(p^{B_2}, p^{R})},
\end{equation}
where again we have replaced $\nu_p(L)$ by $L$, and similarly for $q_1,q_2$.

It is easy to see that
\begin{equation}
\label{eq:Ichi1chi2ConjugationSymmetry}
 \overline{I_{\chi_1,\chi_2}^{(p)}(p^{B_1}, p^{B_2})} = I_{\chi_1,\chi_2}^{(p)}(p^{B_2}, p^{B_1}),
\end{equation}
and one may also easily verify
\begin{equation}
\label{eq:Ichi1chi2CharacterSwitchingSymmetry}
 I_{\chi_1,\chi_2}^{(p)}(p^{B_1}, p^{B_2}) = I_{\chi_2,\chi_1}^{(p)}(p^{L-B_1}, p^{L-B_2}).
\end{equation}

To prove Lemma \ref{lemma:InnerProductEvaluation}, we need three key facts.  First, we claim that
$I^{(p)}_{\chi_1,\chi_2}(p^{B_1}, p^{B_2})$ is unchanged under the replacements $B_i \rightarrow B_i - \min(B_1, B_2)$.  In other words, if we let $B_i = (B_1, B_2) B_i'$, then 
\begin{equation}
\label{eq:Ichi1chi2B1B2B1'B2'}
I_{\chi_1, \chi_2}(B_1,B_2;N) = I_{\chi_1, \chi_2}(B_1', B_2';N). 
\end{equation}
 This matches a corresponding formula for cusp forms (see \cite[p.473]{BlomerMilicevic}).  Secondly, we claim that for $B \geq 1$, we have
\begin{equation}
\label{eq:IpPrimePower}
I^{(p)}_{\chi_1, \chi_2}(1,p^B) = \frac{\lambda_{\chi_1,\chi_2}(p^B)}{p^{\frac{B}{2}}} - \psi(p) \frac{ \lambda_{\chi_1,\chi_2}(p^{B-2})}{p^{\frac{B+2}{2}}},
\end{equation}
where $\psi = \chi_1 \overline{\chi_2}$ is the nebentypus of $E_{\chi_1, \chi_2}$.  Finally, we claim
\begin{equation}
\label{eq:IpNoBs}
I^{(p)}_{\chi_1, \chi_2}(1,1) = 
\begin{cases}
(1-p^{-1}), \qquad &p|(q_1, q_2) \\
1, \qquad &p|q_1, p \nmid q_2 \\
1, \qquad &p|q_2, p \nmid q_1 \\
(1+p^{-1}), \qquad &p \nmid q_1 q_2.
\end{cases}
\end{equation}
 Taking these three facts for granted momentarily, we finish the proof.  The only apparent discrepancy is that if $p|(q_1, q_2)$, then the denominator in \eqref{eq:Achi1chi2PrimePowerFormula} does not seem to agree with \eqref{eq:IpNoBs} when $p|(q_1, q_2)$.  However, in this case, $\lambda(p^B) = 0$, so there is agreement after all.
 
All three facts follow from a more careful evaluation of $I_{\chi_1, \chi_2}^{(p)}(p^{B_1}, p^{B_2})$.
 If $B_i \leq R$ then within \eqref{eq:GipR} this means $e=0$ and $d = R-B_i$, while if $B_i \geq R$ then this means $d=0$ and $e= B_i-R$.  
Hence, 
\begin{multline}
\label{eq:IpBigExpression}
 I_{\chi_1,\chi_2}^{(p)}(p^{B_1},p^{B_2}) 
= 
\sum_{0 \leq R \leq \min(B_1, B_2)} 
\frac{\varphi((p^{q_2+ R}, p^{q_1 +L-R}))}{(p^{q_2+ R}, p^{q_1 +L-R})} 
\frac{\chi_2(p^{B_1-R}) \overline{\chi_2}(p^{B_2-R})}{p^{(B_1-R)(1/2+it) + (B_2-R)(1/2-it)}}
\\
+
\sum_{B_1 < R \leq B_2} \frac{\varphi((p^{q_2+ R}, p^{q_1 +L-R}))}{(p^{q_2+ R}, p^{q_1 +L-R})}
\frac{\chi_1(p^{R-B_1}) \overline{\chi_2}(p^{B_2-R})}{p^{(R-B_1)(1/2+it) + (B_2-R)(1/2-it)}} 
\\
+
\sum_{B_2 < R \leq B_1} \frac{\varphi((p^{q_2+ R}, p^{q_1 +L-R}))}{(p^{q_2+ R}, p^{q_1 +L-R})}
\frac{\chi_2(p^{B_1-R}) \overline{\chi_1}(p^{R-B_2})}{p^{(B_1-R)(1/2+it) + (R-B_2)(1/2-it)}} 
\\
+
\sum_{\max(B_1, B_2) < R \leq L} 
\frac{\varphi((p^{q_2+ R}, p^{q_1 +L-R}))}{(p^{q_2+ R}, p^{q_1 +L-R})} 
\frac{\chi_1(p^{R-B_1}) \overline{\chi_1}(p^{R- B_2})}{p^{(R-B_1)(1/2+it) + (R-B_2)(1/2-it)}}
.
\end{multline}
Of course, at least one of the two middle terms above is an empty sum.
%
%

We first deal with the easiest cases with $p|q_1 q_2$, where we show
\begin{equation}
\label{eq:IpformulapDividesq1q2}
 I^{(p)}_{\chi_1, \chi_2}(p^{B_1}, p^{B_2}) =
 \begin{cases}
   (1-p^{-1}) \delta_{B_1=B_2}, \qquad &p |(q_1, q_2), \\
   \chi_2(p^{B_1-B_2}) 
   p^{\min(B_1,B_2) - \frac{B_1 + B_2}{2} - it(B_1-B_2)},
   \qquad &p |q_1, p \nmid q_2, \\
   \chi_1(p^{B_2-B_1}) p^{\min(B_1,B_2) - \frac{B_1 + B_2}{2} - it(B_2-B_1)}
   \qquad &p | q_2, p \nmid q_1.
 \end{cases}
\end{equation}
Here the expression $\chi_2(p^{B_1 - B_2})$ is interpreted to be $\overline{\chi_2}(p^{B_2-B_1})$ in case $B_2 > B_1$, and similarly for $\chi_1$.
\begin{proof}[Proof of \eqref{eq:IpformulapDividesq1q2}]
 For $p | (q_1, q_2)$, all the terms in \eqref{eq:IpBigExpression}  vanish except $R = B_1 = B_2$, giving the claimed formula.  One may read off the local version of \eqref{eq:Ichi1chi2B1B2B1'B2'} in case $p|q_1 q_2$.
 
 In case $p|q_1$, $p \nmid q_2$ then the second, third, and fourth lines of \eqref{eq:IpBigExpression} vanish, and so we obtain the claimed formula by evaluating the geometric series.
 
 The case $p | q_2$, $p \nmid q_1$ may be derived from the previous case by using \eqref{eq:Ichi1chi2CharacterSwitchingSymmetry}.  
\end{proof}


Now assume that $p \nmid q_1 q_2$.  We claim that
\begin{equation}
\label{eq:IpEvaluationB1equalsB2}
 I^{(p)}_{\chi_1, \chi_2}(p^{B_1}, p^{B_2}) = (1+ p^{-1}), \quad \text{if} \quad B_1 = B_2,
\end{equation}
and if $B_1 < B_2$, then $I^{(p)}_{\chi_1, \chi_2}(p^{B_1}, p^{B_2})$ equals
\begin{equation}
\label{eq:IpEvaluationB1lessthanB2}
\frac{\overline{\chi_2}(p^{B_2-B_1})}{p^{(B_2 - B_1)(1/2-it)}}
+ 
\frac{\overline{\chi_2}(p^{B_2-B_1})}{p^{(B_2 - B_1)(1/2-it)}} (1-p^{-1}) \sum_{j=1}^{B_2-B_1-1} 
\frac{(\chi_1 \chi_2)(p^{j}) }{p^{2it j}}
+
\frac{\chi_1(p^{B_2-B_1})}{p^{(B_2-B_1)(1/2+it)}}.
\end{equation}
The case $B_2 < B_1$ may be derived from \eqref{eq:IpEvaluationB1lessthanB2}, using \eqref{eq:Ichi1chi2ConjugationSymmetry} or \eqref{eq:Ichi1chi2CharacterSwitchingSymmetry}.
In all cases, we see that
\begin{equation}
 I_{\chi_1, \chi_2}^{(p)}(p^{B_1}, p^{B_2}) = 
 I_{\chi_1, \chi_2}^{(p)}(p^{B_1- \min(B_1,B_2)}, p^{B_2- \min(B_1,B_2)}),
\end{equation}
and consequently, we obtain the {\bf first key fact}, \eqref{eq:Ichi1chi2B1B2B1'B2'}.

It is a pleasant coincidence that \eqref{eq:IpEvaluationB1lessthanB2} agrees with \eqref{eq:IpformulapDividesq1q2} for $B_1 < B_2$.

\begin{proof}[Proofs of \eqref{eq:IpEvaluationB1equalsB2} and \eqref{eq:IpEvaluationB1lessthanB2}]
 For the terms in \eqref{eq:IpBigExpression}  with $R \leq \min(B_1, B_2)$, we obtain
\begin{equation}
\label{eq:RatmostMin}
\frac{\chi_2(p^{B_1-B_2})}{p^{\frac{B_1 + B_2}{2} + it(B_1 - B_2)}} \sum_{0 \leq R \leq \min(B_1, B_2)} 
\frac{\varphi((p^{ R}, p^{L-R}))}{(p^{R}, p^{L-R})} 
  p^R.
\end{equation}
If $\min(B_1, B_2) < L$, then \eqref{eq:RatmostMin} simplifies as
\begin{equation*}
\frac{\overline{\chi_2}(p^{B_2-B_1}) p^{\min(B_1,B_2)}}{p^{\frac{B_1 + B_2}{2} - it(B_2 - B_1)}}.
\end{equation*}
If $B_1 = B_2 = L$, then \eqref{eq:RatmostMin} becomes
\begin{equation*}
\frac{\chi_2(p^{B_1-B_2})}{p^{\frac{B_1 + B_2}{2} + it(B_1 - B_2)}}
(p^{L-1} + p^{L}) = (1+ p^{-1}).
\end{equation*}
For the terms in \eqref{eq:IpBigExpression}  with $B_1 < R \leq B_2$, we obtain
\begin{equation*}
\frac{\overline{\chi_2}(p^{B_2-B_1})}{p^{(B_2 - B_1)(1/2-it)}} \sum_{B_1 < R \leq B_2} \frac{\varphi((p^{R}, p^{L-R}))}{(p^{R}, p^{L-R})}
\frac{(\chi_1 \chi_2)(p^{R-B_1}) }{p^{2it(R-B_1)}}.
\end{equation*}
If $B_2 < L$ this simplifies as
\begin{equation*}
\frac{\overline{\chi_2}(p^{B_2-B_1})}{p^{(B_2 - B_1)(1/2-it)}} (1-p^{-1}) \sum_{j=1}^{B_2-B_1} 
\frac{(\chi_1 \chi_2)(p^{j}) }{p^{2it j}},
\end{equation*}
while if $B_2 = L$, it instead equals
\begin{equation*}
\frac{\overline{\chi_2}(p^{B_2-B_1})}{p^{(B_2 - B_1)(1/2-it)}} 
\Big(
 (1-p^{-1}) \sum_{j=1}^{B_2-B_1-1} 
\frac{(\chi_1 \chi_2)(p^{j}) }{p^{2it j}}
+
\frac{(\chi_1 \chi_2)(p^{B_2-B_1}) }{p^{2it (B_2 - B_1)}}
\Big)
.
\end{equation*}
Finally, for the terms with $\max(B_1, B_2) < R \leq L$, we obtain
\begin{equation*}
\frac{1}{p} \frac{\chi_1(p^{B_2-B_1})}{p^{it(B_2-B_1)}} \frac{p^{\min(B_1,B_2)}}{p^{\frac{B_1+B_2}{2}}}.
\end{equation*}

Combining everything, we obtain \eqref{eq:IpEvaluationB1equalsB2} in case $B_1 = B_2$.
Similarly, in case $B_1 < B_2 = L$, then we obtain \eqref{eq:IpEvaluationB1lessthanB2}.
If $B_1 < B_2<L$, then we obtain \eqref{eq:IpEvaluationB1lessthanB2} after some simplifications, taking the term $j=B_2 - B_1$ out from the inner sum.
\end{proof}


%


We deduce the {\bf third key fact} \eqref{eq:IpNoBs}, from \eqref{eq:IpformulapDividesq1q2} and \eqref{eq:IpEvaluationB1equalsB2}.

Recalling the definition \eqref{eq:lambdachi1chi2Def}, 
it is not difficult to derive the {\bf second key fact} \eqref{eq:IpPrimePower} from
\eqref{eq:IpEvaluationB1lessthanB2}.  This completes the proof.
%
%
%
%
\end{proof}

\subsection{An alternative orthonormal basis}
Blomer and Mili\'cevi\'c \cite[Lemma 9]{BlomerMilicevic} constructed an orthonormal basis of the oldclass $S_{t_j, \psi_0}(L;f^*)$ where $f^*$ is a cuspidal newform of level $M$ which is $L^2$ normalized with the level $N$ Petersson inner product, and $\psi_0$ denotes the principal character (see \cite[Lemma 3.15]{Humphries} for arbitrary nebentypus).  Their basis takes the form $\{ f^{(g)} : g|L \}$, where $f^{(g)} = \sum_{d|g} \xi_{g}(d) f^* \vert_{d}$, where $\xi_g(d)$ are certain arithmetical functions defined in terms of the Hecke eigenvalues of $f^*$.
They check that the functions $f^{(g)}$ are orthonormal by expanding bilinearly, calculating $\langle f^* \vert_d, f^* \vert_{d'} \rangle$, for each $d,d' |L$, and evaluating the sums.  Therefore, the same process shows that with the coefficients $\xi_{g}(d)$ defined as for cusp forms, using the Hecke eigenvalues, then the linear combinations of normalized $E_{\chi_1, \chi_2}(dz)$'s also form an orthonormal basis for $\mathcal{E}_{t,\psi_0}(L;E_{\chi_1,\chi_2})$.  The crucial fact here is that Lemma \ref{lemma:InnerProductEvaluation} has the same form as \cite[Lemma 3.15]{Humphries}.

\section{Atkin-Lehner operators}
\label{section:AtkinLehnerEigenvalues}
In this section, we explain how the $E_{\chi_1, \chi_2}(Bz,s)$  and $D_{\chi_1, \chi_2,f}(z,s, \psi)$ behave under the Atkin-Lehner operators.  
\subsection{Newforms}
Essentially everything in this section was worked out by Weisinger \cite{Weisinger} in the holomorphic setting.  

Suppose that $QR=N$, and $(Q,R) = 1$, and define an Atkin-Lehner operator by
\begin{equation*}
 W_Q = \begin{pmatrix} 
        Qr & t \\ Nu & Qv
       \end{pmatrix},
\end{equation*}
where $r,t,u,v \in \mathbb{Z}$, $t \equiv 1 \pmod{Q}$, $r \equiv 1 \pmod{R}$, and $Qrv - R ut = 1$ (so $\det(W_Q) = Q$).  The paper \cite{AtkinLi} is a good reference for these operators.
The nebentypus $\psi$ factors uniquely as $\psi = \psi^{(Q)} \psi^{(R)}$ where $\psi^{(Q)}$ has modulus $Q$ and $\psi^{(R)}$ has modulus $R$.  
Weisinger \cite[p.31]{Weisinger} showed that if $f$ is $\Gamma_0(N)$-automorphic with nebentypus $\psi$, then $f \vert_{W_Q}$, which is independent of $r,t,u,v$, is $\Gamma_0(N)$-automorphic with nebentypus $\overline{\psi}^{(Q)} \psi^{(R)}$.  Here
\begin{equation*}
 f\vert_{W_Q} (z) := j(W_Q, z)^{-k} f(W_Q z).
\end{equation*}

Moreover, Weisinger showed in essence that
\begin{equation}
\label{eq:WeisingerAtkinLehnerPseudoeigenvalue}
 E_{\chi_1, \chi_2} \vert_{W_Q} = c(Q) E_{\chi_1',\chi_2'},
\end{equation}
where the pseudo-eigenvalue $c(Q)$ is an explicit constant depending on the $\chi_i$ (see \eqref{eq:AtkinLehnerEigenvalueFormula} below for a formula), and where the $\chi_i'$ are defined as follows.  Write $\chi_i = \chi_i^{(Q)} \chi_i^{(R)}$, and let $\chi_1' = \chi_2^{(Q)} \chi_1^{(R)} $ and $\chi_2' = \chi_1^{(Q)} \chi_2^{(R)}$.  Note that $\chi_1' \chi_2' = \chi_1 \chi_2$, and that $\chi_1' \overline{\chi_2'} = \overline{\psi}^{(Q)} \psi^{(R)}$.
Actually, Weisinger worked, in effect, with the completed Eisenstein series $E_{\chi_1, \chi_2}^*$ which affects the calculation of the pseuo-eigenvalue, since one must take into account the Gauss sum which appears in \eqref{eq:Echi1chi2FunctionalEquation}.

\begin{proof}[Proof of \eqref{eq:WeisingerAtkinLehnerPseudoeigenvalue}]
 We produce a proof of \eqref{eq:WeisingerAtkinLehnerPseudoeigenvalue}
 which is of an elementary character, and somewhat different in flavor to that of Weisinger's thesis.

 Write $q_i = q_i^{(Q)} q_i^{(R)}$, where $q_1^{(Q)} q_2^{(Q)} = Q$ and $q_1^{(R)} q_2^{(R)} = R$.  Define
\begin{equation*}
 q_1' = q_2^{(Q)} q_1^{(R)}, \qquad \text{and} \qquad
 q_2' = q_1^{(Q)} q_2^{(R)},
\end{equation*}
and observe that $q_i'$ is the modulus of $\chi_i'$.  

From the definition \eqref{eq:EthetaDef}, we have
\begin{equation}
\label{eq:EthetaSlashWQ}
E_{\theta}(z,s) \vert_{W_Q} =j(W_Q, z)^{-k} \sum_{\gamma \in \Gamma_{\infty} \backslash \Gamma_0(1)} \theta(\gamma)
j(\gamma, \sigma W_Q z)^{-k}
 \text{Im}(\gamma \sigma W_Q z)^s,
\end{equation}
where recall $\sigma z = q_2 z$.  Let $\sigma' = (\begin{smallmatrix} \sqrt{q_2'} & \\ & 1/\sqrt{q_2'} \end{smallmatrix})$, so $\sigma' z = q_2' z$.  It is straightforward to check
\begin{equation*}
 \sigma W_Q = \bpm \sqrt{Q} & \\ & \sqrt{Q} \epm
 \underbrace{\bpm q_2^{(Q)} r & q_2^{(R)} t \\ q_1^{(R)} u & q_1^{(Q)} v \epm}_{\lambda}
 \sigma',
\end{equation*}
where observe $\lambda \in SL_2(\mz)$.  One can also show directly that $j(\lambda, \sigma' z) = j(W_Q, z)$, and so
\begin{equation*}
 j(W_Q, z)^{-k} j(\gamma \lambda^{-1}, \sigma W_Q z)^{-k} = 
 j(\gamma, \sigma' z)^{-k}.
\end{equation*}
Therefore, by changing variables $\gamma \rightarrow \gamma \lambda^{-1}$ in \eqref{eq:EthetaSlashWQ}, we obtain
\begin{equation*}
E_{\theta}(z,s) \vert_{W_Q} = \sum_{\gamma \in \Gamma_{\infty} \backslash \Gamma_0(1)} \theta(\gamma \lambda^{-1})
j(\gamma, \sigma' z)^{-k}
 \text{Im}(\gamma \sigma' z)^s,
\end{equation*}
and so now our task is to understand $\theta(\gamma \lambda^{-1})$.  With $\gamma = (\begin{smallmatrix} a & b \\ c & d \end{smallmatrix})$, by direct calculation, we have
$\theta(\gamma \lambda^{-1}) = \chi_1(c q_1^{(Q)} v - d q_1^{(R)} u) \chi_2(- c q_2^{(R)} t + d q_2^{(Q)} r)$.  Using the factorizations $\chi_i = \chi_i^{(Q)} \chi_i^{(R)}$, we obtain
\begin{equation*}
 \theta(\gamma \lambda^{-1})
 =
 \chi_1'(c) \chi_2'(d) \chi_1^{(Q)}(-q_1^{(R)} u)
 \chi_1^{(R)}(q_1^{(Q)} v)
 \chi_2^{(Q)}(-q_2^{(R)} t)
 \chi_2^{(R)}(q_2^{(Q)} r).
\end{equation*}
Next we use that $\chi_1 = \psi \chi_2$  to simplify the above expression, getting
\begin{equation*}
 \theta(\gamma \lambda^{-1}) = \chi_1'(c) \chi_2'(d) \chi_1^{(Q)}(-1) \psi^{(Q)}(q_1^{(R)} u) \psi^{(R)}(q_1^{(Q)} v)
 \chi_2^{(Q)}(-R ut) \chi_2^{(R)}(Q rv).
\end{equation*}
To simplify further, we note
\begin{equation*}
 \chi_2^{(Q)}(-R ut) \chi_2^{(R)}(Q rv) =  \chi_2(Qrv-Rut) =  1 ,
\end{equation*}
using the determinant equation. 
Moreover, from $t \equiv 1 \pmod{Q}$, we have $u \equiv - \overline{R} \pmod{Q}$, so
$\psi^{(Q)}(q_1^{(R)} u) = \overline{\psi}^{(Q)}(- q_2^{(R)})$, and likewise $\psi^{(R)}(q_1^{(Q)} v) = \overline{\psi}^{(R)}(q_2^{(Q)})$.  In all, this discussion shows
\begin{equation*}
 \theta(\gamma \lambda^{-1}) = \theta'(\gamma) 
 \chi_2^{(Q)}(-1) \overline{\psi}^{(Q)}(q_2^{(R)}) \overline{\psi}^{(R)}(q_2^{(Q)}),
\end{equation*}
where $\theta'$ corresponds to $\chi_1',\chi_2'$.

 In all, we obtain that
\begin{equation}
\label{eq:AtkinLehnerEigenvalueFormula}
 E_{\chi_1, \chi_2}( z,s) \vert_{W_Q} = 
\chi_2^{(Q)}(-1)
\overline{\psi}^{(Q)}(q_2^{(R)})
\overline{\psi}^{(R)}(q_2^{(Q)})
 E_{\chi_1 ', \chi_2 '}(z,s). \qedhere
\end{equation}
\end{proof}

\subsection{The Fricke involution}
As a particularly important special case of \eqref{eq:AtkinLehnerEigenvalueFormula}, if $Q = N$, then we obtain
\begin{equation*}
 E_{\chi_1, \chi_2}( z,s) \vert_{W_N} = 
\chi_2^{}(-1) E_{\chi_2, \chi_1}(z,s).
\end{equation*}
It is a slightly subtle point that $W_N$ is not exactly the same operator as the Fricke involution $\omega_N := (\begin{smallmatrix} 0 & -1 \\ N & 0 \end{smallmatrix})$.  Indeed, we have
\begin{equation*}
 W_N = \begin{pmatrix} Nr & t \\ Nu & Nv \end{pmatrix}
 = \underbrace{\begin{pmatrix} -t & r \\ -N v & u \end{pmatrix}
 }_{\gamma_N \in \Gamma_0(N)} \omega_N,
\end{equation*}
and note $\psi(\gamma_N) = \psi(-1)$.
For a $\Gamma_0(N)$-automorphic function $f$ of nebentypus $\psi$, we have 
\begin{equation*}
 f(\omega_N z) = f(\gamma_N^{-1} W_N z) = j(\gamma_N^{-1}, W_N z)^{k} \psi(-1) j(W_N, z)^k f\vert_{W_N} = \psi(-1) j(\omega_N, z)^{k} f \vert_{W_N}.
\end{equation*}
%
Collecting these formulas, we obtain
\begin{equation*}
 E_{\chi_1, \chi_2}\Big(\frac{i}{q_1 q_2 y},s \Big) = i^{-k} \chi_2(-1) E_{\chi_2, \chi_1}(iy,s).
\end{equation*}
For the completed Eisenstein series, we may derive
\begin{equation*}
 E_{\chi_1, \chi_2}^*\Big(\frac{i}{q_1 q_2 y},s \Big) = q_1^{1/2-s} q_2^{s-1/2}  i^{-k+\delta_1-\delta_2} \chi_2(-1)
 \epsilon(\chi_1) \epsilon(\overline{\chi_2})
 E_{\chi_2, \chi_1}^*(iy,s).
\end{equation*}
Needless to say, this is compatible with \eqref{eq:FrickeFormula1}, which had more restrictive conditions on $k$ and the parity of the characters.

\subsection{Oldforms}
For this subsection, we restrict attention to the trivial nebentypus case with $k=0$, and $\chi_1 = \chi_2 = \chi$ of modulus $\ell_1 = \ell_2 = \ell$.
Viewing $E_{\chi,\chi}(Bz)$ as on $\Gamma_0(\ell^2 M)$ with $B | M$, we need to see how it operates under the larger collection of Atkin-Lehner operators for this subgroup.  Suppose $q$ is a prime so that $q^{\alpha} || M \ell^2$, $q^{\beta}|| M$, (so $q^{\alpha-\beta} || \ell^2$), and $q^{\gamma} || B$.  If $\gamma \leq \frac{\beta}{2}$, then \cite[Lemma 26]{AtkinLehner}  showed
\begin{equation*}
 (E_{\chi,\chi} \vert_{B} ) \vert_{W_q} = (E_{\chi,\chi} \vert_{W_q'} ) \vert_{B'},
\end{equation*}
where we now describe what this means.  First, $f \vert_{B} = f(Bz)$.  The operator $W_q$ is the Atkin-Lehner involution for the group $\Gamma_0(\ell^2 M)$, while $W_q'$ is the one associated to $\Gamma_0(\ell^2)$.  Finally, $B'$ is defined by setting $B = q^{\gamma} B_0$ where $q \nmid B_0$, and then $B' = q^{\beta-\gamma} B_0$.
Thus, we have for $\gamma \leq \frac{\beta}{2}$ that
\begin{equation}
\label{eq:WqAppliedToEchichiBz}
 (E_{\chi,\chi}(Bz,s) ) \vert_{W_q}  = \chi^{(q)}(-1) E_{\chi,\chi}(B'z,s),
\end{equation}
where $q^j || \ell^2$ (so $j = \alpha - \beta$).
If $\gamma > \frac{\beta}{2}$, then the same formula holds, as can be proved by doing the same calculation for $E_{\chi,\chi}(q^{\beta-\gamma} B_0 z, s)$.

 For each prime $q$ dividing $M \ell^2$, the map $B \rightarrow B'$ is an involution on the set of divisors of $M$.  Note that if $q|\ell$ but $q \nmid M$, then $B' = B$. 

Thus, the Atkin-Lehner operators permute the functions $E_{\chi,\chi}(Bz, s)$, with a multiplication by $\chi^{(q)}(-1)$.  The corresponding property for cusp forms was important in \cite{KY}, showing that the the Fourier coefficients of $f \vert_{B}$ at an Atkin-Lehner cusp are essentially the same as at infinity.

It may be interesting to mention that the Atkin-Lehner operators also permute the functions $D_{\chi,f}(z,s) := D_{\chi,\chi,f}(z,s,1)$, since this is a desirable property of an orthonormal basis.
\begin{myprop}
\label{prop:AtkinLehnerPermutesDchi}
 Let $f = \ell g$, with $g | M$, and let $W_q$ be the Atkin-Lehner involution on $\Gamma_0(\ell^2 M)$ with $q | \ell^2 M$.  
  Suppose $q^{j} || \ell^2$.  
 Then
 \begin{equation*}
   D_{\chi, \ell g} \vert_{W_q} = \chi^{(q)}(-1) D_{\chi, \ell g'}.
 \end{equation*}
\end{myprop}
Remark. If $q | \ell$ but $q \nmid M$, then $g' = g$, and the claimed formula follows immediately from \eqref{eq:WqAppliedToEchichiBz}, since $B' = B$.

\begin{proof}
We have
\begin{equation}
\label{eq:AtkinLehnerAppliedToDchi}
(M \ell)^s \frac{L(2s,\chi^2 \chi_{0,N})}{L(2s,\chi^2)} D_{\chi, \ell g} =  (g, M/g)^s 
 \sum_{\substack{a|g}}  \sum_{\substack{b | \frac{M}{g} }}
 \frac{\mu(a) \mu(b)}{(ab)^{s}}
  \chi(ab)  
 E_{\chi,\chi}\Big(\frac{bg}{a } z, s\Big).
\end{equation}
Let $\Delta$ denote the right hand side of \eqref{eq:AtkinLehnerAppliedToDchi}.  Then by \eqref{eq:WqAppliedToEchichiBz},  $ \Delta \vert_{W_{q}}$ has the same expression but with $(bg/a)$ replaced by $(bg/a)'$, and multiplied by $\chi^{(q)}(-1)$.

By abuse of notation, write $M = q^{M} M_0$, $g = q^{g} g_{0}$, and within the sum we write $a = q^a a_0$ and $b = q^b b_0$.  Then we have 
\begin{multline*}
 \Delta = (g_{0}, M_0/g_{0})^s 
 \sum_{\substack{a_0|g_{0}}}  \sum_{\substack{b_0 | \frac{M_0}{g_{0}} }}
 \frac{\mu(a_0) \mu(b_0)(q^{g}, q^{M-g})^s }{(a_0 b_0)^{s}}
 \chi(a_0 b_0)
 \\
 \sum_{0 \leq a \leq g} \sum_{0 \leq b \leq M - g}
 \frac{\mu(q^a ) \mu(q^b) }{(q^a q^b )^{s}}
  \chi(q^a  q^b )   
 E_{\chi,\chi}\Big(q^{b+g-a} \frac{b_0 g_{0}}{a_0 } z, s\Big).
\end{multline*}
Similarly,
\begin{multline*}
 \Delta \vert_{W_q} = \chi^{(q)}(-1) (g_{0},  M_0/g_{0})^s 
 \sum_{\substack{a_0|g_{0}}}  \sum_{\substack{b_0 | \frac{M_0}{g_{0}} }}
 \frac{\mu(a_0) \mu(b_0)(q^{g}, q^{M-g})^s }{(a_0 b_0)^{s}}
 \chi(a_0 b_0)
 \\
 \sum_{0 \leq a \leq g} \sum_{0 \leq b \leq M - g}
 \frac{\mu(q^a ) \mu(q^b) }{(q^a q^b )^{s}}
  \chi(q^a  q^b )     
 E_{\chi,\chi}\Big((q^{b+g-a})' \frac{b_0 g_{0}}{a_0 } z, s\Big).
\end{multline*}
 According to the discussion preceding Proposition \ref{prop:AtkinLehnerPermutesDchi}, $(q^{b+g-a})' = q^{M - (b+g-a)}$, and $(q^{g})' = q^{M - g}$, or in additive notation, $g' = M - g$.  Finally, switching the roles of $a$ and $b$, and applying the substitution $g' = M - g$ completes the proof.
\end{proof}

\section{Bruggeman-Kuznetsov for newforms}
\label{section:Kuznetsov}
In this section, we use some of the material developed in this paper to 
give a Bruggeman-Kuznetsov formula for newforms, which
extends the newform Petersson formula derived in \cite{PetrowYoung}.

\subsection{Statement of Bruggeman-Kuznetsov}
Suppose that $u_j$ form an orthonormal basis of Hecke-Maass cusp forms of level $N$ and nebentypus $\psi$, and write
\begin{equation*}
u_j(z)= \sum_{n \neq 0} \rho_j(n) e(nx) 
                  W_{0,i t_j}(4 \pi |n| y),
             \end{equation*}
             where
             \begin{equation*}
              W_{0,it_j}(4 \pi y) = 2 \sqrt{y} K_{it_j} (2 \pi y) .
             \end{equation*}       
Similarly, write
\begin{equation*}
 E_{\a}(z,s, \psi) = \delta_{\a=\infty} y^{s} + \rho_{\a, \psi}(s) y^{1-s} +  \sum_{n \neq 0} \rho_{\a, \psi}(n, s) e(nx) W_{0,s-\frac12}(4 \pi |n| y).
\end{equation*}
Renormalize the coefficients by defining
\begin{equation*}
 \nu_{j}(n) = \left(\frac{4\pi |n| }{\cosh(\pi t_j) }\right)^{1/2} \rho_{j} (n), \qquad
		\nu_{\a,t}(n) = \left(\frac{4\pi |n| }{\cosh(\pi t ) }\right)^{1/2} \rho_{\a, \psi}  (n,1/2+it). 
\end{equation*}
If $u_j$ is a newform we have $\nu_j(n) = \nu_j(1) \lambda_j(n)$, where $\lambda_j(n)$ are the Hecke eigenvalues.

The Bruggeman-Kuznetsov formula for $mn>0$ reads as
\begin{multline}
\label{eq:Kuznetsov}
 \sum_j \nu_j(m) \overline{\nu_j(n)}   h(t_j) + \sum_{\a} \frac{1}{4 \pi} \intR \nu_{\a,t}(m) \overline{\nu_{\a,t}}(n)  h(t) dt
 \\
 = \delta_{m=n} g_0  + \sum_{c \equiv 0 \shortmod{N}} \frac{S_{\psi}(m,n;c)}{c} g\Big(\frac{4 \pi \sqrt{mn}}{c}\Big),
\end{multline}
where
\begin{equation*}
 g_0 = \frac{1}{\pi} \intR t  \tanh(\pi t) h(t) dt, \qquad
 g(x) = 2i \intR \frac{J_{2it}(x)}{\cosh(\pi t)} t h(t) dt,
\end{equation*}
and
\begin{equation*}
 S_{\psi}(m,n;c) = \sumstar_{x \shortmod{c}} \psi(x) e\Big(\frac{xm + \overline{x} n}{c} \Big).
\end{equation*}

We may wish to only choose an orthogonal basis of cusp forms instead of an orthonormal basis;  the formula is modified by dividing by $\langle u_j, u_j \rangle$, for then $\frac{\nu_j(m) \overline{\nu_j(n)} }{\langle u_j, u_j \rangle}$ is invariant under re-scaling.  The inner product is
\begin{equation*}
 \langle u_j, u_j \rangle = \int_{\Gamma_0(N) \backslash \mathbb{H}} |u_j(z)|^2 \frac{dx dy}{y^2}.
\end{equation*}
This normalization explains why the diagonal term on the right hand side of \eqref{eq:Kuznetsov} does not grow with $N$.

Next we discuss how \eqref{eq:Kuznetsov} changes if we choose an alternative basis of Eisenstein series.  Let $\{F \}$ be an orthogonal basis for the space of Eisenstein series, with inner product defined formally as in Section \ref{section:orthogonality}.  Here we view the spectral parameter $t$ as held fixed, so the dimension of this space equals the number of singular cusps for $\psi$.

We claim the quantity
\begin{equation*}
 \sum_{F \text{ orthogonal basis}} \frac{\nu_{F,t}(m) \overline{\nu_{F,t}(n)} }{\langle F, F \rangle_{\text{Eis}}}
\end{equation*}
is independent of the orthogonal basis.  This follows from Proposition \ref{prop:SpectralFormulasChangeOfBasis}, since one may interpret 
this expression as the part of $\langle P_n, P_m \rangle$ coming from the continuous spectrum, for some generalized Poincare series (or integrals thereof).  

Hence, we may re-phrase the Bruggeman-Kuznetsov formula using the decomposition into newforms as in \eqref{eq:EisensteinAtkinLehnerDecomposition}.  That is, we have
\begin{multline}
\label{eq:DeltaNinfinityDef}
 \sum_{\a} \frac{1}{4 \pi} \intR \nu_{\a,t}(m) \overline{\nu_{\a,t}}(n)  h(t) dt
 \\
 = \sum_{LM=N}
 \sum_{E_{\chi_1, \chi_2} \in \mathcal{H}_{t,\psi}^*(M)} 
 \frac{1}{4 \pi} \intR \sum_{\substack{F \text{ orthogonal basis} \\ \text{for } \mathcal{E}_t(L;E_{\chi_1,\chi_2}) }}
 \frac{\nu_{F,t}(m) \overline{\nu_{F,t}(n)} }{\langle F, F \rangle_{\text{Eis}}} h(t) dt.
\end{multline}
This is analogous with the decomposition of cusp forms into newforms, which gives
\begin{equation}
\label{eq:DeltaN0Def}
 \sum_j \nu_j(m) \overline{\nu_j(n)}   h(t_j)
 = \sum_{LM=N} \sum_{f \in H_{t_j, \psi}^*(M)} 
 \sum_{\substack{F \text{ orthogonal basis} \\ \text{for } S_{t_j,\psi}(L;f) }} \frac{\nu_F(m) \overline{\nu_F(n)}}{\langle F, F \rangle}.
\end{equation}
The formula \eqref{eq:DeltaN0Def} is not original; one may alternatively consult \cite[(2.11)]{BlomerHarcosMichel} or \cite[(7.32)]{KnightlyLi}.

\subsection{Bruggeman-Kuznetsov for newforms, squarefree level}
\label{section:KuznetsovNewforms}
For the rest of this section, suppose $N$ is square-free, and $\psi$ is the principal character.
Let us fix a function $h$, which we suppress from the notation, and define $\Delta_N(m,n)$ as the left hand side of \eqref{eq:Kuznetsov}.
Also, write $\Delta_N = \Delta_{N,0} +\Delta_{N,\infty}$ corresponding to the cuspidal part and the Eisenstein part, separately (so $\Delta_{N,0}$ equals \eqref{eq:DeltaN0Def}, and $\Delta_{N,\infty}$ equals \eqref{eq:DeltaNinfinityDef}).  Let $\Delta_{N,0}^*$ denote the newform analog of $\Delta_{N,0}$, where the Maass forms are restricted to be newforms of level $N$.
That is, set
\begin{equation*}
 \Delta_{N,0}^*(m,n) = \sum_{t_j} h(t_j) \sum_{f \in H_{t_j}^*(M)} \frac{\nu_f(m) \overline{\nu_f(n)}}{\langle f, f\rangle}.
\end{equation*}

A nearly-identical proof to that in \cite{PetrowYoung} gives a formula for $\Delta_{N,0}^*$ in terms of $\Delta_{M,0}$'s and vice-versa.  
Precisely, with $\tt{x} = 0$, we have
\begin{multline}
\label{eq:KuznetsovOldforms}
 \Delta_{N,\tt{x}}(m,n) = \sum_{LM=N}\frac{1}{\nu(L)} \sum_{\ell | L^{\infty}}  \frac{\ell}{\nu(\ell)^2 } \sum_{d_1, d_2 | \ell} c_{\ell}(d_1) c_{\ell}(d_2) 
 \sum_{\substack{u|(m,L) \\ v | (n,L)}} \frac{u v }{(u, v)} \frac{\mu(\frac{u v}{(u, v)^2})}{\nu(\frac{u v}{(u, v)^2})}
 \\
\sum_{\substack{a | (\frac{m}{u}, \frac{u}{(u, v)}) \\ b | (\frac{n}{v}, \frac{v}{(u, v)})}}   
\sum_{\substack{e_1 | (d_1, \frac{m}{a^2 (u,v)}) \\ e_2 | (d_2, \frac{n}{b^2 (u,v)})} }
 \Delta_{M,\tt{x}}^*\Big(
\frac{m d_1}{a^2 e_1^2 (u, v)}, 
 \frac{n d_2}{b^2 e_2^2 (u, v)}\Big),
\end{multline}
where a definition of the coefficients $c_{\ell}(d)$ may be found in \cite{PetrowYoung} (we have no need of them here).  This formula
is inverted by
\begin{multline}
 \label{eq:KuznetsovOldformsInverted}
 \Delta_{N,\tt{x}}^*(m,n) = \sum_{LM=N}\frac{\mu(L)}{\nu(L)} \sum_{\ell | L^{\infty}}  \frac{\ell}{\nu(\ell)^2 } \sum_{d_1, d_2 | \ell} c_{\ell}(d_1) c_{\ell}(d_2) 
 \sum_{\substack{u|(m,L) \\ v | (n,L)}} \frac{u v }{(u, v)} \frac{\mu(\frac{u v}{(u, v)^2})}{\nu(\frac{u v}{(u, v)^2})}
 \\
\sum_{\substack{a | (\frac{m}{u}, \frac{u}{(u, v)}) \\ b | (\frac{n}{v}, \frac{v}{(u, v)})}}   
\sum_{\substack{e_1 | (d_1, \frac{m}{a^2 (u,v)}) \\ e_2 | (d_2, \frac{n}{b^2 (u,v)})} }
 \Delta_{M,\tt{x}}\Big(
\frac{m d_1}{a^2 e_1^2 (u, v)}, 
 \frac{n d_2}{b^2 e_2^2 (u, v)}\Big).
\end{multline}
The proof of \cite{PetrowYoung} only uses Hecke theory, and so all the arguments carry through without any substantial changes.

The goal of the rest of this section is to show that \eqref{eq:KuznetsovOldforms} and \eqref{eq:KuznetsovOldformsInverted} hold equally well for the continuous spectrum, that is, with $\tt{x}=\infty$.  Then since $\Delta_N = \Delta_{N,0} + \Delta_{N,\infty}$, and likewise for $\Delta_M^*$, we obtain analogous formulas for $\Delta_N$ and $\Delta_M^*$.

When $N$ is square-free and $\psi$ is principal, then the decomposition \eqref{eq:EisensteinAtkinLehnerDecomposition2}
simplifies since $\mathcal{H}_t^*(M) = \{0 \}$ for $M \neq 1$, and $\mathcal{H}_t^*(1)$ is the level $1$ Eisenstein series, $E(z,1/2+it)$.
So, define $\Delta_{M,\infty}^*(m,n) = 0$ unless $M=1$, in which case $\Delta_{1,\infty}^*(m,n) = \Delta_{1,\infty}(m,n)$.

For the analysis of $\Delta_{N,\infty}$ to proceed in parallel with that of $\Delta_{N,0}$, we need to pick a basis for $\mathcal{E}_t(L;E)$ analogous to the one chosen in \cite{PetrowYoung}.  Instead of the $E_{\a}$ or $D_{\chi,\ell}$, we start with $E$, the level $1$ Eisenstein series.  Let $\phi$ be a function defined on the divisors of $N$, satisfying $\phi(p) = \pm 1$, extended multiplicatively.  There are $\tau(N)$ such functions $\phi$, since $N$ is squarefree.  Then define
\begin{equation}
\label{eq:EphiDef}
 E_{\phi} = \sum_{d|N} \phi(d) E\vert_{W_d},
\end{equation}
which is on $\Gamma_0(N)$.  Here $W_d$ is the Atkin-Lehner involution, and from \eqref{eq:WqAppliedToEchichiBz}, we have $(E \vert_{W_d})(z, 1/2 + it) = E(dz,1/2+it)$.
It follows from \eqref{eq:EphiDef} that $E_{\phi} \vert_{W_d} = \phi(d) E_{\phi}$, and so these functions form an orthogonal basis for $\mathcal{E}_t(N;E)$.  Thus, 
\begin{equation*}
 \Delta_{N,\infty}(m,n) = \intR h(t) 
 T_t(m,n) dt,
 \qquad T_t(m,n) = 
 \sum_{\phi} \frac{\nu_{E_{\phi},t}(m) \overline{\nu_{E_{\phi},t}(n)} }{\langle E_{\phi}, E_{\phi} \rangle}.
\end{equation*}
Now we wish to evaluate $T_t(m,n)$ in an analogous way to \cite{PetrowYoung}.

As in \cite[(2.5)]{PetrowYoung}, we have
\begin{equation*}
 \langle E_{\phi}, E_{\phi} \rangle = \tau(N) \sum_{d|N} \phi(d) \langle E \vert_{W_d}, E \rangle.
\end{equation*}
Next we claim
\begin{equation}
\label{eq:AbbessUllmoGeneralizationToEisenstein}
 \langle E_t\vert_{W_d}, E_t \rangle = \frac{\tau_{it}(d) \sqrt{d}}{\nu(d)}
 \langle E_t, E_t \rangle,
\end{equation}
where $\nu(d) = \prod_{p|d} (p+1)$, which is the index of $\Gamma_0(d)$ in $\Gamma_0(1)$ for square-free $d$.
When $E$ is replaced by a cuspidal newform, then this was proved by Abbes and Ullmo \cite{AbbesUllmo}.  More general inner product calculations may be found in Section \ref{section:InnerProduct}, but the case here is brief enough that a direct evaluation is desirable.

Returning to Theorem \ref{thm:EchichiInTermsofEa}, we see that
\begin{equation*}
 E(Bz,s) = N^s \sum_{\substack{d|A \\ e|B}} \frac{1}{(de)^s} E_{\frac{1}{Bd/e}}(z,s).
\end{equation*}
Therefore,
\begin{equation*}
 \langle E_t \vert_{W_B}, E_t \rangle = N
 \sum_{a|N} \frac{1}{a^{1/2+it}} \sum_{\substack{b|B, c|\frac{N}{B}}} \frac{1}{(bc)^{1/2-it}} \langle E_{\frac{1}{a}}, E_{\frac{1}{Bc/b}} \rangle.
\end{equation*}
Here the inner product vanishes unless $a = \frac{B}{b} c$.  Thus we obtain
\begin{equation*}
 \frac{1}{4 \pi} \langle E_t \vert_{W_B}, E_t \rangle = N
 \Big(\sum_{b|B} \frac{1}{(B/b)^{1/2+it} b^{1/2-it}} \Big)
 \Big(\sum_{c|\frac{N}{B}} \frac{1}{c} \Big) = \tau_{it}(B) \sqrt{B}\nu(N/B). 
\end{equation*}
Taking $B=1$, we get $\frac{1}{4 \pi} \langle E_t, E_t \rangle = \nu(N)$, and so \eqref{eq:AbbessUllmoGeneralizationToEisenstein} follows, as well as 
\begin{equation}
\label{eq:EinnerproductComparisonTwoGroups}
\langle E_t, E_t \rangle_N = \nu(N) \langle E_t, E_t \rangle_1
\end{equation}
where the subscript on the inner product symbol denotes the level of the group to which the inner product is attached.
Hence
\begin{equation*}
 \langle E_{\phi}, E_{\phi} \rangle = \tau(N) \langle E, E \rangle \prod_{p|N} \Big(1 + \frac{\phi(p) \tau_{it}(p) p^{1/2}}{\nu(p)} \Big),
\end{equation*}
which is the analog of \cite[(2.6)]{PetrowYoung}.

By a direct calculation with the Fourier expansion, we have
\begin{equation*}
 \nu_{E_{\phi}}(m) = \sum_{u  | (m,N)} \phi(u) u^{1/2} \nu_E(m/u) = \nu_E(1) \sum_{u  | (m,N)} \phi(u) u^{1/2} \lambda_E(m/u),
\end{equation*}
where $\lambda_E(n) = \tau_{iT}(n)$; this is the analog of \cite[(2.8)]{PetrowYoung}.
We therefore need to evaluate the inner sum over $\phi$, namely
\begin{equation}
\label{eq:chiinnerproductsum}
T_t(m,n)
= \frac{1}{\tau(N) \langle E, E \rangle}_N \sum_{\phi} \overline{\lambda_{E_{\phi}}(m)} \lambda_{E_{\phi}}(n) \prod_{p | N} \Big(1 + \frac{\phi(p) \lambda_E(p)p^{1/2}}{\nu(p)} \Big)^{-1},
\end{equation}
where we have used \eqref{eq:AbbessUllmoGeneralizationToEisenstein}.  Here \eqref{eq:chiinnerproductsum} is analogous to \cite[(3.2)]{PetrowYoung}.  At this point, all the calculations of $T_t(m,n)$ run completely parallel to those in \cite[Section 3]{PetrowYoung}, since the formulas that were used there are: Hecke relations, \eqref{eq:AbbessUllmoGeneralizationToEisenstein}, and \eqref{eq:EinnerproductComparisonTwoGroups}, which are the same in both cases of cusp forms vs. Eisenstein.

Therefore, \eqref{eq:KuznetsovOldforms} holds with $\tt{x}=\infty$.
We also claim that the inversion formula \eqref{eq:KuznetsovOldformsInverted} holds with $\tt{x}=\infty$.
The key to this is that the inversion formula proved in \cite[Section 4]{PetrowYoung} is a combinatorial formula proved by inclusion-exclusion and does not depend on any properties of $\Delta_{M,\infty}^*$.  
Similarly, the intermediate hybrid formulas appearing in \cite[Section 5]{PetrowYoung} also extend to the Bruggeman-Kuznetsov formula; these only rely in the previous formulas and Hecke relations which hold equally well in both the Maass and Eisenstein cases.


One may then set up the hybrid cubic moment for Maass forms as in \cite[(8.7)]{PetrowYoung}: there exist positive weights $\omega_{u_j}$ and $\omega_t$ so that we define
\begin{multline*}
 \mathcal{M}(r,q) = \sum_{\substack{u_j \text{ new, level $rq'$} \\ q'| \tilde{q}}} \omega_{u_j} h(t_j) L(1/2, u_j \otimes \chi_q)^3 
 \\
 + 
 \frac{1}{4\pi} \intR
 \sum_{\substack{E_t \text{ new, level $rq'$} \\ q'| \tilde{q}}}  \omega_{t} h(t) L(1/2, E_t \otimes \chi_q)^3 dt.
\end{multline*}
Actually this sum over $E_t$ is empty except when $r=1$ and $q'=1$ in which case the problem reduces to the one treated in \cite{ConreyIwaniec}.  There is no reason to exclude $r=1$, however.  Now one may approach $\mathcal{M}(r,q)$ as in \cite{PetrowYoung}, just as the original paper of Conrey-Iwaniec \cite{ConreyIwaniec} dealt equally well with Maass forms and holomorphic forms.

\section{Proof of the inversion formula}
\label{section:Inversion}
\begin{proof}[Proof of Lemma \ref{lemma:InversionElementaryLemma}]
 Write  
 \begin{equation*}
  \mathcal{K} = \mathop{\sum_{d|A} \sum_{e|B}}_{(d,e) = 1} \omega_1(d) \omega_2(e) J\Big( \frac{Bd}{e}\Big),
 \end{equation*}
so the desired identity is  $\mathcal{K} = K(B ) \prod_{p|L} (1-\omega_1(p) \omega_2(p))$.
Inserting the definition \eqref{eq:JintermsofK}, we have
 \begin{equation*}
  \mathcal{K} =  \sum_{\substack{d|A \\ e|B \\ (d,e) = 1}} 
 \sum_{\substack{a | \frac{Bd}{e} }}
  \sum_{\substack{b | \frac{A e}{d} }}
 \mu(a) \mu(b) \omega_1(bd) \omega_2(ae)
 K\Big(\frac{b d B}{ae} \Big).
 \end{equation*}
Reversing the orders of summation, we obtain
\begin{equation*}
\mathcal{K} = \sum_{a|L} \sum_{b | L}
 \sum_{
  \substack{d | A, Bd \equiv 0 \shortmod{ae} \\
    e | B, Ae \equiv 0 \shortmod{bd} 
   \\
   (d,e) = 1}}
 \mu(a) \mu(b)\omega_1(bd) \omega_2(ae)
 K\Big(\frac{b d B}{ae}  \Big).
\end{equation*}
 Now let $(a,d) = g$ and $(b,e) = h$ and write $d = gd'$ and $e= he'$, giving
\begin{equation*}
\mathcal{K} = \sum_{a|L} \sum_{b | L}
\sum_{g | (a,A)} \sum_{h | (b,B)}
 \sum_{
  \substack{d' | \frac{A}{g}, Bd' \equiv 0 \shortmod{\frac{a}{g} h e'} \\
    e' | \frac{B}{h}, A  e' \equiv 0 \shortmod{ \frac{b}{h} g d'} \\
    (d', \frac{a}{g}) = 1, \thinspace
    (e', \frac{b}{h}) = 1 \\
    (gd', he') = 1
    }}
 \mu(a) \mu(b) 
 \omega_1(b g d')
 \omega_2(a h e')
 K\Big(\frac{b g d' B}{ah e'}  \Big).
\end{equation*}
 Now since $(d', \frac{a}{g} h e') = 1$, the first congruence condition is equivalent to $\frac{a}{g} h e' | B$ (which automatically implies $e'| \frac{B}{h}$, so this latter condition may be omitted since it is redundant).  Similarly, the second congruence is equivalent to $\frac{b}{h} g d' | A$.  Next we move the sums over $g$ and $h$ to the outside, and define $a = ga'$, $b = hb'$.  Simplifying, we obtain
\begin{equation*}
\mathcal{K} = \sum_{g | A} 
\sum_{h | B}
\sum_{a'|\frac{L}{g}} \sum_{b' | \frac{L}{h}}
 \sum_{
  \substack{ a' h e' | B \\
    b' g d' |A   \\
    (d', a') = 1\\
    (e', b') = 1 \\
    (gd', he') = 1
    }}
  \mu(ga') \mu(hb')
 \omega_1(b' g h  d')
 \omega_2( a'  g h  e')
 K\Big(\frac{ b'  d' B}{ a' e'}  \Big).
\end{equation*}
 Next expand $\mu(ga') = \mu(g) \mu(a')$, recording the coprimality condition $(a',g) = 1$, and similarly for $\mu(hb')$.  Then
\begin{equation*}
\mathcal{K} = \sum_{
  \substack{ a' h e' | B \\
    b' g d' |A   \\
    (d'g, a') = 1\\
    (e'h, b') = 1 \\
    (gd', he') = 1
    }}
 \mu(g) \mu(a') \mu(h) \mu(b') 
 \omega_1(b' g h  d')
 \omega_2( a'  g h  e')
 K\Big(\frac{ b'  d' B}{ a' e'}  \Big).
\end{equation*}
Now let $(a',b') = r$, and write $a'=r a''$, $b'=r b''$ where now $(a'', b'') = 1$.  Then
\begin{equation*}
\mathcal{K} = \sum_{
  \substack{ a'' r h e' | B \\
    b'' r g d' |A   \\
    (\dots)
    }}
 \mu(g) \mu(a'') \mu^2(r) \mu(h) \mu(b'') 
 \omega_1(b'' r g h  d')
 \omega_2( a'' r  g h  e')
 K\Big(\frac{ b''  d' B}{ a'' e'}  \Big),
\end{equation*}
where $(\dots)$ represents the following coprimality conditions:
\begin{equation}
\label{eq:greatbiglistofcoprimalityconditions}
 (d'g, a'' r) = 1, \quad
    (e'h, b'' r) = 1, \quad
    (gd', he') = 1, \quad
    (a'', b'') = 1, \quad
    (a'' b'', r) = 1.
\end{equation}

Now let $a'' e' = \alpha$ be a new variable, and likewise $b'' d' = \beta$; the coprimality conditions in \eqref{eq:greatbiglistofcoprimalityconditions} translate into these conditions: $(\alpha, \beta) =1$, $(\alpha, rg) = 1$, $(\beta,  r h) = 1$, and $(r,gh) = (g,h) = 1$.  Note that $a''$ and $b''$ do not occur in this list of conditions.
Thus we obtain
\begin{equation*}
\mathcal{K} = \sum_{
  \substack{  r h \alpha | B \\
     r g \beta |A   \\
     (\dots)
    }}
 \mu(g)  \mu^2(r) \mu(h)  
 \omega_1( r g h  \beta)
 \omega_2(  r  g h  \alpha)
 K\Big(\frac{ \beta B}{ \alpha} \Big)
\sum_{a'' | \alpha} \sum_{b'' | \beta} 
 \mu(a'') \mu(b''),
\end{equation*}
where now $(\dots)$ represents the coprimality conditions translated into the new variables.
The upshot is that M\"obius inversion gives $\alpha = \beta = 1$, and we obtain  $\mathcal{K} = \kappa K(B)$, where
\begin{equation*}
\kappa := \sum_{
  \substack{  r h  | B, \thinspace 
     r g  |A, \\
    (r, gh) = 1, \thinspace
    (g, h) = 1
    }}
 \mu(g)  \mu^2(r) \mu(h)  (\omega_1 \omega_2)(r g h ).
\end{equation*}
We claim that $\kappa = \prod_{p|AB} (1- \omega_1(p) \omega_2(p))$, which may be checked prime-by-prime by brute force.
\end{proof}


\section{Mellin transform of Whittaker function}
\label{section:WhittakerMellinTransform}
In this section, we take the opportunity to correct \cite[Lemma 8.2]{DFI}.  The overall method of \cite{DFI} is valid, but a typo early in the derivation makes it difficult to correct the mistake without going through the entire process again.

Recall the definition
\begin{equation*}
 \Phi_k^{\varepsilon}(s,\beta)
 = \sqrt{\pi} \int_0^{\infty} 
 \Big(W_{\frac{k}{2}, \beta}(4  y) + \varepsilon \frac{\Gamma(\beta+\frac{1+k}{2})}{\Gamma(\beta+\frac{1-k}{2})} W_{-\frac{k}{2}, \beta}(4  y)
 \Big) y^{s-\frac12} \frac{dy}{y},
\end{equation*}
where $\varepsilon = \pm 1$, and that we wish to show
\begin{equation*}
 \Phi_k^{\varepsilon}(s,\beta)
 =   p_k^{\varepsilon}(s,\beta)
 \Gamma\Big(\frac{s+\beta + \frac{1- \varepsilon (-1)^k}{2}}{2}\Big)
 \Gamma\Big(\frac{s-\beta + \frac{1-\varepsilon}{2}}{2}\Big),
\end{equation*}
 where  $p_k^{\varepsilon}(s,\beta)$ is a certain polynomial defined recursively below.  
 
The first step of the derivation in \cite{DFI} is to give two recursion formulas for the Whittaker function, the first of which contains a typo.  The corrected formulas are (cf. \cite[(13.15.10), (13.15.12)]{DLMF})
\begin{align*}
 \frac{2 \beta}{\sqrt{y}} W_{\alpha,\beta}(y) 
 &= W_{\alpha+\frac12, \beta+\frac12}(y) - W_{\alpha+\frac12,\beta-\frac12}(y) \\
 &= (\beta-\alpha+\tfrac12) W_{\alpha-\frac12, \beta+\frac12}(y) + (\beta + \alpha - \tfrac12) W_{\alpha-\frac12, \beta-\frac12}(y).
\end{align*}


Next define
\begin{equation}
 V_{k,\beta}^{\varepsilon}(y)
 =
 W_{\frac{k}{2}, \beta}(4  y) + 
 \varepsilon
 \frac{\Gamma(\beta+\frac{1+k}{2})}{\Gamma(\beta+\frac{1-k}{2})} W_{-\frac{k}{2}, \beta}(4  y).
\end{equation}
The above recursion formulas give
\begin{multline*}
 \frac{2 \beta}{\sqrt{y}} V_{k,\beta}^{\varepsilon}(\frac{y}{4})
 =
 (\beta + \tfrac{1-k}{2}) W_{\frac{k-1}{2}, \beta+\frac12}(y) 
 + 
 (\beta + \tfrac{k-1}{2}) W_{\frac{k-1}{2}, \beta-\frac12}(y) 
 \\
 +
 \varepsilon
 \frac{\Gamma(\beta+\frac{1+k}{2})}{\Gamma(\beta+\frac{1-k}{2})}
 \Big(
 W_{\frac{1-k}{2}, \beta+\frac12}(y)
 - W_{\frac{1-k}{2}, \beta-\frac12}(y)
 \Big),
\end{multline*}
and using $\Gamma(s+1) = s \Gamma(s)$, we derive
\begin{equation*}
 \frac{ \beta}{\sqrt{y}} V_{k,\beta}^{\varepsilon}(y)
 = (\beta + \tfrac{1-k}{2}) V_{k-1, \beta+\frac12}^{\varepsilon}(y)
 +
 (\beta + \tfrac{k-1}{2}) V_{k-1, \beta-\frac12}^{-\varepsilon}(y),
\end{equation*}
which replaces \cite[(8.28)]{DFI}.  Then on integration, we derive
\begin{equation*}
 \Phi_{k}^{\varepsilon}(s,\beta)
 =(1 - \tfrac{k-1}{2 \beta}) \Phi_{k-1}^{\varepsilon}(s+\tfrac12, \beta + \tfrac12)
 +
 (1 + \tfrac{k-1}{2 \beta}) \Phi_{k-1}^{-\varepsilon}(s+\tfrac12, \beta - \tfrac12),
\end{equation*}
for $\beta \neq 0$.  When $k=0$, then the formula for $V_{0,\beta}^{\varepsilon}(y)$ given in \cite[p.532]{DFI} is correct, and by \cite[(6.561.16)]{GR}, we derive
\begin{equation*}
 \Phi_0^{\varepsilon}(s,\beta) =  \Big(\frac{1+\varepsilon}{2} \Big) \Gamma\Big(\frac{s+\beta}{2}\Big) \Gamma\Big(\frac{s-\beta}{2}\Big),
\end{equation*}
which differs from \cite[(8.30)]{DFI} by a factor $\frac14$.  This shows the claimed formula for $\Phi_0^{\varepsilon}$ with $p_0^{\varepsilon} = \frac{1+\varepsilon}{2}$.

Now proceed by induction on $k \geq 1$.  We obtain
\begin{multline*}
 \Phi_{k}^{\varepsilon}(s,\beta)
 =(1 - \tfrac{k-1}{2 \beta}) p_{k-1}^{\varepsilon}(s+\tfrac12, \beta + \tfrac12) \Gamma\Big(\frac{s+\beta + 1+ \frac{1- \varepsilon (-1)^{k-1}}{2}}{2}\Big)
 \Gamma\Big(\frac{s-\beta + \frac{1-\varepsilon}{2}}{2}\Big)
 \\
 +
 (1 + \tfrac{k-1}{2 \beta}) p_{k-1}^{-\varepsilon}(s+\tfrac12, \beta - \tfrac12)
 \Gamma\Big(\frac{s+\beta + \frac{1+ \varepsilon (-1)^{k-1}}{2}}{2}\Big)
 \Gamma\Big(\frac{s-\beta + 1+ \frac{1+\varepsilon}{2}}{2}\Big).
\end{multline*}
Next we use
\begin{equation*}
 \Gamma\Big(\frac{s+\beta + 1+ \frac{1- \varepsilon (-1)^{k-1}}{2}}{2}\Big)
 = \Gamma\Big(\frac{s+\beta +  \frac{1- \varepsilon (-1)^{k}}{2}}{2}\Big) 
 \times
 \begin{cases}
   1, \quad \varepsilon  = -(-1)^k \\
   \frac{s + \beta}{2}, \quad \varepsilon  = (-1)^k.
 \end{cases}
\end{equation*}
And similarly,
\begin{equation*}
 \Gamma\Big(\frac{s-\beta + 1+ \frac{1+\varepsilon}{2}}{2}\Big)
 = \Gamma\Big(\frac{s-\beta +  \frac{1- \varepsilon}{2}}{2}\Big) 
 \times
 \begin{cases}
   1, \quad \varepsilon  = -1 \\
   \frac{s - \beta}{2}, \quad \varepsilon  = 1.
 \end{cases}
\end{equation*}
Therefore, we obtain
a recursion
\begin{equation}
\begin{split}
 p_{k}^{\varepsilon}(s,\beta)
 = &(1 - \tfrac{k-1}{2 \beta}) p_{k-1}^{\varepsilon}(s+\tfrac12, \beta + \tfrac12) \times
 \left\{ \begin{array}{lr}
   1, & \varepsilon  = -(-1)^k \\
   \frac{s + \beta}{2}, & \varepsilon = (-1)^k
 \end{array}
 \right\}
 \\
 + &
 (1 + \tfrac{k-1}{2 \beta}) p_{k-1}^{-\varepsilon}(s+\tfrac12, \beta - \tfrac12)
 \times
 \left\{ \begin{array}{lr}
   1, & \varepsilon  = -1 \\
   \frac{s - \beta}{2}, & \varepsilon  = 1
 \end{array}
 \right\}.
 \end{split}
\end{equation}
 By direct calculation, we have
 \begin{align*}
  p_1^{+}(s,\beta) & = 1& p_1^{-}(s,\beta) & = 1 \\
  p_2^{+}(s,\beta) & = s-\tfrac12 & p_2^{-}(s,\beta) & = 2 \\
  p_3^{+}(s,\beta) & = 2s-1 -\beta  & p_3^{-}(s,\beta) & = 2s-1 +\beta \\
  p_4^{+}(s,\beta) & = 2(s-\tfrac12)^2 - \beta^2 + \tfrac14  & p_4^{-}(s,\beta) & = 4(s-\tfrac12).
 \end{align*}
Compared with \cite[(p. 533)]{DFI}, the sign on $\beta$ is reversed.

\end{document}